\documentclass[12pt,amstex]{amsart}

\usepackage{epsfig}
\usepackage{amsmath}
\usepackage{amssymb}
\usepackage{amscd}
\usepackage{graphicx}
\usepackage{enumerate}
\usepackage{pstricks}

\usepackage{mathptmx}
\usepackage{mathrsfs}

\usepackage{verbatim}
\usepackage{url}
\usepackage[all]{xy}
\usepackage{color}

\usepackage[colorlinks=true,citecolor=blue]{hyperref}

\usepackage{stmaryrd}
\usepackage{epsfig}
\usepackage{amsmath}
\usepackage{amssymb}

\usepackage{amscd}
\usepackage{graphicx}
\usepackage{pstricks}

\newtheorem{thm}{Theorem}[section]
\newtheorem{lem}[thm]{Lemma}

\newtheorem{prop}[thm]{Proposition}

\newtheorem{cor}[thm]{Corollary}
\theoremstyle{definition}
\newtheorem{de}[thm]{Definition}

\newtheorem{conj}{Conjecture}
\theoremstyle{remark}
\theoremstyle{Conjecture}

\numberwithin{equation}{section}

\def \N {\mathbb N}

\def \Z {\mathbb Z}

\def \F {\mathcal F}
\def \G {\mathcal{G}}
\def \U {\mathcal U}
\def \V {\mathcal V}

\def \X {\mathcal{X}}
\def \Y {\mathcal{Y}}
\def \O {\mathcal{O}}
\def \Q {\mathbf{Q}}
\def \E {\mathbb{E}}

\def \cl {{\rm cl}}

\def \lra {\rightarrow}

\def \ra {\rightarrow}

\def \ep {\epsilon}
\def \d {\delta}
\def \D {\Delta}

\def \G {\mathcal G}
\def \id {{\rm id}}

\def \cl {{\rm cl}}

\def \intt {{\rm int}}

\def \lra {\longrightarrow}
\def \ra {\rightarrow}

 \newcommand{\RP}{\textbf{RP}}
 \newcommand{\AP}{\textbf{AP}}

\begin{document}
\title{Topological characteristic factors and nilsystems}

\author{Eli Glasner$^1$}
\address{$^1$Department of Mathematics, Tel Aviv University, Tel Aviv, Israel}
\email{glasner@math.tau.ac.il}
\author{Wen Huang}
\author{Song Shao}
\author{Benjamin Weiss$^2$}
\address{$^2$Institute of Mathematics, Hebrew University of Jerusalem, Jerusalem, Israel}
\email{weiss@math.huji.ac.il}
\author{Xiangdong Ye}
\address{CAS Wu Wen-Tsun Key Laboratory of Mathematics, and
Department of Mathematics, University of Science and Technology of China, Hefei, Anhui, 230026, P.R. China}

\email{wenh@mail.ustc.edu.cn}
\email{songshao@ustc.edu.cn}
\email{yexd@ustc.edu.cn}


\thanks{This research is supported by National Natural
Science Foundation of China (11971455, 11731003, 11571335, 11431012).}

\date{June 1, 2016}
\date{April, 2019}
\date{Feb. 6, 2020}
\date{May 10, 2020}
\date{May, 30, 2020}
\date{June 2, 2020}
\date{June 19, 2020}

\begin{abstract}
We prove that the maximal infinite step pro-nilfactor $X_\infty$ of a minimal
dynamical
system $(X,T)$ is the
topological characteristic factor in a certain sense. 
Namely, we show that by an almost one to one modification of $\pi:X \rightarrow X_\infty$, the
induced open extension
$\pi^*:X^* \rightarrow X^*_\infty$ has the following property: for $x$ in a dense $G_\d$ set of $X^*$, the orbit closure
$L_x=\overline{\O}((x,x,\ldots,x), T\times T^2\times \ldots \times T^d)$
is $(\pi^*)^{(d)}$-saturated, i.e.
$L_x=((\pi^*)^{(d)})^{-1}(\pi^*)^{(d)}(L_x)$.


Using results derived from the above fact, we
are able to
answer several open questions:
(1) if $(X,T^k)$ is minimal for some $k\ge 2$, then for any $d\in \N$ and any $0\le j<k$ there is a sequence
$\{n_i\}$ of $\Z$ with $n_i\equiv j\ (\text{mod}\ k)$ such that $T^{n_i}x\rightarrow x, T^{2n_i}x\rightarrow x, \ldots, T^{dn_i}x\rightarrow x$
for $x$ in a dense $G_\delta$ subset of $X$;
(2) if $(X,T)$ is totally minimal, 
then $\{T^{n^2}x:n\in\Z\}$ is dense in $X$ for $x$ in a dense $G_\delta$ subset of $X$; 
(3) for any $d\in\N$ and any minimal system, 
which is an open extension of its maximal distal factor, $\RP^{[d]}=\AP^{[d]}$,
where the latter is the regionally proximal relation of order $d$ along arithmetic progressions.
\end{abstract}

\maketitle

\tableofcontents


\section{Introduction}

In this section we will first
provide some background
related to characteristic factors, present some open questions,
and then state our main results and explain the main ideas of the proofs.

\subsection{Backgrounds}\

\medskip

\subsubsection{Characteristic factors}\

\medskip

A connection between ergodic theory and additive combinatorics was established
 in the 1970's with Furstenberg's elegant proof of
Szemer\'edi's theorem
via  ergodic theory. Furstenberg \cite{F77} proved
Szemer\'edi's  theorem
by means of the following theorem: let $T$ be
a measure preserving transformation (m.p.t. for short) on the Borel probability space
$(X,\X,\mu)$, then for every $d \ge 1$ and $A\in \mathcal{X}$ with positive
measure,
\begin{equation}\label{liminf-Fur}
    \liminf_{N\to \infty} \frac{1}{N}\sum_{n=0}^{N-1}
    \mu(A\cap T^{-n}A\cap T^{-2n}A\cap \ldots \cap T^{-dn}A)>0.
\end{equation}
In view of this theorem
it is natural to ask about the convergence of these averages; or
more generally, about the convergence, either in $L^2(X,\mu)$ or pointwise, of the {\em multiple
ergodic averages}
\begin{equation}\label{m-ave}
 \frac 1 N\sum_{n=0}^{N-1}f_1(T^nx)\ldots
f_d(T^{dn}x),
\end{equation} where $f_1, \ldots , f_d \in L^\infty(X,\mu)$.
After nearly 30 years' efforts of many researchers, this problem (for $L^2$ convergence) was
finally solved in \cite{HK05, Z}.

In the study of
the avarages
(\ref{m-ave}), the idea of characteristic factors plays an
important role.
For the
origin of these ideas and this terminology, see \cite{F77} and \cite{FW96}. To be more precise,
let $(X,\X,\mu, T)$ be a
measure preserving transformation
(m.p.t.)
 and $(Y,\Y,\nu, T)$ be a factor of $X$. For $d\ge 1$, we say that $Y$ is a
{\em characteristic factor} of $X$  if for all $f_1,\ldots,f_d\in L^\infty(X,\mu)$,
\begin{equation*}
\frac{1}{N}\sum_{n=0}^{N-1} f_1(T^nx) \ldots f_d(T^{dn}x)  - \frac{1}{N}\sum_{n=0}^{N-1} \E(f_1|\Y)(T^nx) \ldots \E(f_d|\Y)(T^{dn}x)\to 0
\end{equation*}
in $L^2(X,\mu)$.

Finding a good characteristic factor for
certain schemes of averages often
yields a reduction of the problem of evaluating their limit behavior.
For example,
Furstenberg \cite{F77} proved for each $d\ge 2$, the $(d-1)$-step measurable distal factor (in the
structure theorem of an ergodic m.p.t.) is a characteristic factor
for
(\ref{m-ave}).
The result in \cite{HK05, Z} improves the result of Furstenberg significantly, i.e. they
show that for each $d\ge 2$,
a $(d-1)$-step pro-nilsystem is a characteristic factor
for (\ref{m-ave}).

\medskip

By
a
{\it topological dynamical system} $(X,T)$ (t.d.s. for short) we mean a
homeomorphism $T$ from a compact metric space $X$ to itself.
A  counterpart of the notion of characteristic factors in a t.d.s. was first studied in 1994 by
Glasner \cite{G94}. There, the author studied the characteristic
factors for the transformation $\tau_d=T\times T^2\times \ldots \times T^d$ in the sense of {\it saturation}:
let $\pi: X\rightarrow Y$ be a map between two sets $X$ and $Y$. A subset $L$ of $X$ is called {\em $\pi$-saturated} if
$\{x\in L: \pi^{-1}(\pi(x))\subset L\}=L$, i.e. $L=\pi^{-1}(\pi(L))$.
Given a factor map $\pi: (X,T)\rightarrow (Y,T)$ and $d\ge 2$,
the t.d.s. $(Y,T)$ is said to be a {\em $d$-step topological
characteristic factor (along $\tau_d$)
of  $(X,T)$}, if there exists a dense $G_\d$
subset $\Omega$ of $X$ such that for each $x\in \Omega$ the orbit
closure $L_x=\overline{\O}((x, \ldots,x), \tau_d)$ is $\pi\times \ldots \times
\pi$ ($d$-times) saturated. 


In \cite{G94}, it was shown that for minimal systems, up to a canonically defined proximal extension,
a characteristic family for $\tau_d$ is the family of canonical PI flows of class $d-1$. In particular,
if $(X,T)$ is distal, then its largest class $d-1$ distal factor (in the structure theorem of Furstenberg
\cite{F63})
is its topological characteristic factor along $\tau_d$. Moreover, if $(X,T)$ is weakly mixing, then the trivial system is
its topological characteristic factor. 

\medskip
As in the ergodic situation, in topological dynamics one expects that the largest class $d-1$ distal factor can be replaced by
the $(d-1)$-step pro-nilfactor. So, based on the result of \cite{G94} and the
parallelism
between
ergodic theory and topological dynamical systems, one naturally asks:

\medskip
\noindent{\bf Question 1:}
{\it Assume that $(X,T)$ is minimal which is a RIC weakly mixing extension of a distal system.
Is it true that its maximal $(d-1)$-step pro-nilfactor
is its topological characteristic factor along $\tau_d$?}


\medskip

(For a technical reason we need to assume here that the extension is RIC, or maybe just open.
As we will see one can always achieve this situation by applying a canonical construction
which in some sense does not change much the original system.)

\medskip

\subsubsection{Odd recurrence}\

\medskip

It is easy to see that one consequence of (\ref{liminf-Fur}) is the following multiple ergodic recurrence theorem (MERT for short):
if $(X,\mathcal{X},\mu, T)$ is a m.p.t., then for each $d\in\N$
and $A\in \mathcal{X}$ with $\mu(A)>0$ there is $n\in\N$ such that
\begin{equation}\label{Fur-linear-ergodic}
\mu(A\cap T^{-n}A\cap \ldots \cap T^{-dn}A)>0.
\end{equation}
As an immediate application of EMRT, one has that if $(X,T)$ is minimal then for each $d\in\N$ and each non-empty open
subset $U$ of $X$, there is $n\in\N$ such that
\begin{equation}\label{Fur-linear-topo}
U\cap T^{-n}U\cap \ldots \cap T^{-dn}U\not=\emptyset.
\end{equation}
We will
refer to this property
as the topological multiple recurrence theorem (TMRT, for short).
For
topological proofs of the TMRT see \cite{FW, F1, BPT, BL96}.
It is easy to see that TMRT is
equivalent to the following statement: if $(X,T)$ is minimal and $d\in\N$, then there is a
dense $G_\delta$ subset $\Omega$ of $X$ such that for each $x\in \Omega$ there is an
increasing sequence $\{n_i\}$ in $\N$ with
\begin{equation}\label{m-recu}
T^{n_i}x\lra x,\ \  T^{2n_i}x\lra x,\ \  \ldots, \ \  T^{dn_i}x\lra x.
\end{equation}

We note that TMRT, or (\ref{m-recu}),
is also equivalent to the well known
Van der Wareden theorem: if $r\in\N$ and
$\N=N_1\cup \ldots \cup N_r$ then one of
the sets $N_i$ contains arbitrarily long arithmetic progressions.

There are
several ways
in which one can generalize (\ref{Fur-linear-ergodic}) and (\ref{Fur-linear-topo}).
The first one is
to extend
these properties
to nilpotent group actions
(there are counterexamples for solvable groups \cite{BL02}). 
For this type of results we refer to \cite{BL96,Leibman94,Leibman98} and the references therein.

Another way is to restrict $n$ to a particular congruence class: $n\equiv j\ (\text{mod}\ k)$
for a given $k\ge 2$ and $0\le j<k$; or
to
other subsets of $\N$, for example
to
the set of primes.
Host and Kra \cite{HK02} (for $d\le 3$) and Frantzikinakis \cite[Corollary 6.5]{Fran04} (for the general $d$)
showed that if $(X,\mathcal{X},\mu, T)$ is a m.p.t. and $T^k$ is ergodic for some $k\ge 2$,
then for any $d\in\N$, any $A\in\mathcal{X}$ with $\mu(A)>0$ and any $0\le j<k$, we have
$\mu(A\cap T^{-n}A\cap \ldots \cap T^{-dn}A)>0,$
for some $n\equiv j\ (\text{mod}\ k)$.

In view of
the results of Host-Kra and Frantzikinakis the following question,
which is well known in the community, was open till now.

\medskip
\noindent{\bf Question 2:}
{\it Let $(X,T^k)$ be minimal for some $k\ge 2$ and $d\in\N$. Is it true that for any non-empty open subset $U$ of $X$ and $0\le j<k$
one has
\begin{equation}\label{Fur-linear-topo-finer}
U\cap T^{-n}U\cap \ldots \cap T^{-dn}U\not=\emptyset,
\end{equation}
for some $n\equiv j\ (\text{mod}\ k)$?}

\medskip
We remark that if $(X,T)$ is minimal and weakly mixing then the Question 2 has an affirmative answer,
see \cite{G94, HSY-19}.
We note that the result can not be obtained by applying \cite{Fran04}, since $(X,T^k)$
is minimal for some $k\ge 2$ does not imply that there is a Borel invariant probability measure $\mu$ with $(X, \mathcal{X}, T^k,\mu)$ ergodic.

\subsubsection{Density problems}\

\medskip

In ergodic theory there are many results stating that the time averages are equal
to the spatial averages under various ergodicity assumptions. For example, the von Neumann mean ergodic theorem tells us that
if $(X,\mathcal{X}, \mu,T)$ is ergodic, then for each $f\in L^2(X,\mu)$, one has $\frac{1}{N}\sum_{n=1}^{N} f(T^nx)\lra \int f d\mu, N\to\infty$ in $L^2(X,\mu)$.
The corresponding topological statement is the following: if $(X,T)$ is a transitive t.d.s., then there is
a dense $G_\delta$ set $\Omega$ of $X$ such that each $x\in \Omega$ has a dense orbit.


Furstenberg \cite{F1} (for $L^2$) and Bourgain \cite{Bo} (pointwisely for general $p$) 
have shown that if
$(X,\mathcal{X}, \mu,T)$ is totally ergodic, then for each $f\in L^p(X,\mu)$ with $p>1$
and each non-constant integral polynomials $P(n)$, we have 
\begin{equation}\label{b-point}
\frac{1}{N}\sum_{n=1}^Nf(T^{P(n)})x\lra \int f d\mu \ \text{in} \ L^p(X,\mu).
\end{equation}

As not every minimal system admits a totally ergodic measure, the following question is natural.

\medskip
\noindent{\bf Question 3:}
{\it Let $(X,T)$ be totally minimal and $P(n)$ be a non-constant integral polynomial. Is it true that
$\{T^{P(n)}x:n\in\Z\}$ is dense in $X$ for $x$ in a dense $G_\delta$ subset of $X$?}

\medskip
We note that the total minimality assumption is necessary for the above question.
Let $X$ be a periodic orbit of period $3$ and then $(X,T^2)$ is minimal
but it is easy to check that
$\{T^{n^2}x:n\in\Z\}$ is not dense in $X$ for any $x\in X$.

\medskip
A more challenging problem is
whether one
can replace the polynomial times by the set of primes in the above question.
{ A convergence similar to (\ref{b-point})
has been proved to be true in ergodic theory due to
Vinogradov \cite{Vin}: under the total ergodicity assumption, for all $f\in L^2(X,\mu)$,
$$\lim_{N\to\infty} \frac{1}{\pi(N)}\sum_{p\le N, p\ {\rm prime}}T^pf=\int fd\mu \quad \text{in} \ L^2(X,\mu),$$
where $\pi(N)$ denotes the number of primes less than or equal to $N$.
See \cite{Bou-87, FranHK07, MR} for more information.}

\subsubsection{Regionally proximal relations of higher order}\

\medskip

Finally we proceed to give the background of the last problem.

The notion of the regionally proximal relation $\RP^{[1]}$ is
an important tool in the study of a t.d.s. For a minimal t.d.s.,
when the acting group is amenable, it is known that it is a closed equivalence relation
and that $X/\RP^{[1]}$ is the maximal equicontinuous factor.
In \cite{HK05} Host and Kra introduced very useful
new
tools, like the so-called Gowers-Host-Kra seminorms,  the $\G^{[d]}$-actions, etc.,
to construct a (pro-nilsystem) factor $Z_{d-1}$, and to show that it is the
characteristic factor for the averages (\ref{m-ave}).

To get the corresponding factors in a t.d.s. in the
pioneering
work
\cite{HKM} Host-Kra-Maass introduced the notion of
the
{\it regionally proximal relation of order $d$}, denoted by $\RP^{[d]}$,
and proved that for a minimal distal
$\Z$-system,
$\RP^{[d]}$ is an equivalence relation and
that
$X/\RP^{[d]}$ is a pro-nilsystem of order $d$.
Later, Shao and Ye \cite{SY} showed that $\RP^{[d]}$ is an equivalence relation for
arbitrary
minimal systems
of abelian groups.
 See Glasner, Gutman and Ye \cite{GGY18} for the
case of
general group actions.

In \cite{HKM} the author study $\RP^{[d]}$ through the so-called {\em dynamical parallelepiped of
dimension $d$}, $\Q^{[d]}(X)$. Since the averages in (\ref{m-ave}) is only related to $\tau_d$, it
is natural to define
a kind of regionally proximal relation of higher order, by using $\tau_d$
directly. In \cite{GHSY} the authors
followed this direction by introducing a notion, called
regionally proximal relation of order
$d$ along arithmetic progressions, denoted by $\AP^{[d]}$.
Among other things, the authors proved that under some additional assumptions,
for a
uniquely ergodic minimal distal
system, one
has $\RP^{[d]}=\AP^{[d]}$ for every $d\in\N$.
A conjecture \cite[Conjecture 1.1]{GHSY} posed there is the following

\medskip
\noindent{\bf Conjecture 1.1 of \cite{GHSY}:}
{\it Let $(X,T)$ be a minimal distal system. Then $\RP^{[d]}=\AP^{[d]}$ for $d\in\N$.}

\subsection{The main results}\

\medskip
In this subsection we state our main results. For a minimal system $(X,T)$ and $d\in\N$
we use $\RP^{[d]}$ to denote the regionally proximal relation of order $d$, and
$\RP^{[\infty]}=\bigcap_{d\ge 1} \RP^{[d]}.$ Let $X_i=X/\RP^{[i]}$, $i\in\N\cup\{\infty\}$.
Then $X_1$ is the maximal equicontinuous factor of $X$.

\medskip
\noindent{\bf Theorem A:} {\it Let $(X,T)$ be a minimal system, and $\pi:X\rightarrow X_\infty$ be the factor map.
Then there are minimal systems $X^*$ and $X_\infty^*$ which are almost one to one
extensions of $X$ and $X_\infty$ respectively, and a commuting diagram below such that  $X_\infty^*$ is a
$d$-step topological characteristic factor of $X^*$ for all $d\ge 2$,
\[
\begin{CD}
X @<{\sigma^*}<< X^*\\
@VV{\pi}V      @VV{\pi^*}V\\
X_\infty @<{\tau^*}<< X_\infty^*
\end{CD}
\]}


\medskip
It is worth
mentioning
that, using Theorem A, we can show that
when
a minimal system $(X,T)$
is an open extension of  its maximal distal factor, then for each $d\in \N$,
the $d$-step topological characteristic factor of $X$ is $X_{d-1}=X/\RP^{[d-1]}$
(Theorem~ \ref{main-distal}).
This fact emphasises the analogy with the ergodic situation.
We
point out that the number $d-1$ is the sharp result, since $T^{n_i}x\ra x, \ldots, T^{(d-1)n_i}x\ra x$ and
$T^{dn_i}x\ra y$ for some $y$ implies $(x,y)\in \RP^{[d-1]}$ (see Lemma \ref{sharpnumber}).
Moreover, since
in the structure of a general minimal system $(X,T)$
there may appear proximal extensions,
in some sense, Theorem A is the best result we can expect,
meaning that we need the almost one to one modifications (see the example in \cite{G94}).

\medskip

Assume that $(X,T)$ is minimal and $x\in X$. The orbit closure of $(x,\ldots,x)$
under the action $\langle\sigma_d, \tau_d\rangle$ is denoted by $N_d(X,T,x)$, where
$$\tau_d(T)=T\times T^2\times \ldots \times T^d, \ \text{and}\ \sigma_d(T)=T^{(d)}=T\times T\times \ldots \times T.$$
It is easy to see that $N_d(X,T,x)$ is independent of $x$, which will be denoted by $N_d(X,T)$ or $N_d(T)$ or $N_d(X)$.
A basic result proved by Glasner \cite{G94} is that
$N_d(X)$ is minimal under the $\langle\sigma_d, \tau_d\rangle$ action.
We note that the minimality of $N_d(X)$ implies van der Wareden's theorem, see \cite[Theorem 1.56]{G-book}.

\medskip

We further investigate the dynamical properties of $N_d(X)$,
and one consequence of this study, namely Theorem C, will be
used in proving Theorems D and E. 

\medskip
\noindent{\bf Theorem B:} \
{\it  Let $(X,T)$ be a minimal system and $d\in\N$. Then
the maximal equicontinuous factor of $(N_d(X,T), \langle\sigma_d, \tau_d\rangle)$
is $(N_d(X_1,T), \langle\sigma_d, \tau_d\rangle)$, where as above $X_1$ is the maximal equicontinuous factor of $(X,T)$.}

\medskip
In fact, we will show more, see Theorems \ref{c-1-1} and \ref{c-1-1-distal}. Namely, it is proved that for each $d,k\in\N$,
the
maximal $k$-step
pro-nilfactor of $N_d$ is the same
as the one of $N_d(X_\infty)$, and 
that
there is a dense $G_\delta$ set $\Omega\subset X$
such that for each $x\in\Omega$, the maximal $k$-step pro-nilfactor of
$\overline{\O}(x^{(d)}, \tau_d)$ is the same
as the one
of $\overline{\O}((\pi_\infty x)^{(d)}, \tau_d)$, where $\pi_\infty:X\ra X_\infty$ is the
canonical
factor map.
Applying Theorem B we have

\medskip

\noindent{\bf Theorem C:}\
{\it Let $(X,T)$ be a minimal system and $k\ge 2$.
Then $(X,T^k)$ is minimal if and only if $N_d(X, T)=N_d(X, T^k)$ for each $d\in\N$.}

\medskip
As applications
of these results
we have an affirmative answer to Question 2
and state it in its equivalence form:

\medskip
\noindent{\bf Theorem D:}
{\it Let $(X,T^k)$ be minimal for some $k\ge 2$ and $d\in\N$. Then for any $d\in \N$ and any $0\le j<k$ there is a sequence
$\{n_i\}$ with $n_i\equiv j\ (\text{mod}\ k)$ such that $T^{n_i}x\rightarrow x, T^{2n_i}x\rightarrow x, \ldots, T^{dn_i}x\rightarrow x,$ for
$x$ in
a dense $G_\d$ subset of $X$.}

\medskip
It is shown in \cite{F77} that if $r\in \N$ and $\N=N_1\cup \ldots \cup N_r$,
then
there is $i$ such that $N_i$ contains a piece-wise syndetic set.
Then the orbit closure of the
characteristic
function $1_{N_i}\in \{0,1\}^{\N}$ contains a point $\omega$ which is not $(0,0,\ldots)$, and
such that
each word appearing in $\omega$ appears syndetically.
We say that the partition is an {\it irreducible
of type $k$} ($k\ge 2$) if each word appearing in $\omega$ also appears in the position $nk+1$
for some $n\in\N$.
Using this terminology Theorem D can be
restated as follows:

If $\N=N_1\cup \ldots \cup N_r$
is an irreducible partition of type $k$, then there is
an
$i$ such that for each $l\in\N$ and $0\le j<k$ there are $a,b\in\N$
with $a,a+b,\ldots,a+lb\in N_i$, and $b\equiv j\ (\text{mod}\ k)$.

\medskip

The following is an affirmative answer to Question 3 for polynomials of degree 2.

\medskip
\noindent{\bf Theorem E:}
{\it Let $(X,T)$ be a totally minimal system, and $P(n)=an^2+bn+c$ be an integral polynomial with
$a\not=0$. Then there is a dense $G_\delta$ subset $\Omega$ of $X$ such that
for
every
$x\in \Omega$, the
set $\{T^{P(n)}(x):n\in\Z\}$ is dense in $X$.}

\medskip

Finally, we confirm Conjecture 1.1 in \cite{GHSY}. In fact we show more, namely:

\medskip
\noindent{\bf Theorem F:}
{\it  Let $(X,T)$ be a minimal system which is an open extension
of its maximal distal factor, then for any $d\in\N$, $\AP^{[d]}=\RP^{[d]}$.}


\subsection{The main ideas of the proofs}\

\medskip
We start from the proof of Theorem A.
In a deep
sense Theorem A is similar to the ergodic case: one wants to reduce
questions regarding
the $\tau_d$-action from
a general system (meaning ergodic m.p.t. or minimal t.d.s.) to
the same questions in a pro-nilsystem.

Now unlike the ergodic situation where the structure theorem for ergodic systems
involves only two kinds of extensions, namely isometric and weakly mixing extensions,
in the structure theorem of the general minimal system,
see \cite{EGS} and \cite{V83},
proximal extensions (which in general need not be open) necessarily appear.
This fact causes great difficulties when one wants to apply this structure theorem.
To overcome these difficulties, we slightly modify the structure by introducing various kinds of
auxiliary
extensions.
If all we need is opennes of the maps then
the price we pay is
the introduction of an auxiliary
almost one to one modification of the original extension.
Fortunately such modification exists in a canonical way by the classical
construction
called the {\em O-diagram}.

\medskip
The second
difficulty
we face is
more
essential,
namely: there are no tools like
Gowers-Host-Kra seminorms or the van der Corput lemma in topological dynamics,
whereas these tools are
frequently used in \cite{HK05}.
The two main ingredients we use
instead
are:
a simplified version of a construction
used by Glasner \cite{G94}, and the
essential
use of the characterizations of the regionally proximal relation of order $d$
obtained
by Huang-Shao-Ye \cite{HSY16},
which involves Poincar\'e and Birkhoff sets, introduced by Furstenberg \cite{F77},
and their higher order versions by Frantzikinakis, Lesigne and Wierdl \cite{FLW}. 

\medskip

Once we have these tools, the real difficulty is
in checking
one
specific
 condition in the construction. Namely, we need to verify that if
$O$ is a relatively open subset of $N_d$, then the orbit closure of
$O$ under the $\tau_d$-action is ``saturated" in the sense
that if it
contains
some point in a fibre, then
it already contains the full fibre
(see Lemma~ \ref{lem-Key}).
In trying to do
this for a while, we realized that
this can be done only
when all the generators
of the group $\langle\tau_d,\sigma_d\rangle$ are used.
For example, when $d=3$,
and
$\langle\sigma_3, \tau_3\rangle=\langle T\times T\times T, T\times T^2\times T^3\rangle$,
in the proof we have to use the generators
$$\{\id\times T\times T^2, \sigma_3\},\{T\times \id\times T^{-1}, \sigma_3\}, \text{and}\ \{T^2\times T\times \id, \sigma_3\}.$$
In previous works
we never
expected
that the last two generators may become
useful.
The idea to use all
the
generators is crucial in the current paper, and we also
believe that this phenomenon
will become useful in other settings as well.

\medskip
Now we turn to the proof of Theorem B. By Theorem A, it is relatively easy to see that the maximal equicontinuous factor of $N_d(X)$
is the same as the one of $N_d(X_\infty)$. So, it remains to
show this for the higher order pro-nilsystems.
This is done
by using Glasner's result (Lemma \ref{qiu}), a recent result proved by Qiu and Zhao (Lemma~ \ref{eli-thm}), and Lemma \ref{sharpnumber}.

Theorem C is
obtained as
an application of Theorem B,
together with a result for equicontinuous systems (Proposition \ref{t-n}), and a
discussion of the decomposition for minimal systems under the iterations of $T$.

\medskip
With the preparations we have
outlined so far
it is not hard to get Theorems D, E and F,
except
that
we need to develope
a tool
in order to switch results for $N_d$ under the
$\langle\tau_d,\sigma_d\rangle$ action to
$N_d$ under the $\tau_d$ action.
We provide such a tool in Lemma \ref{elibenjy}.

\medskip
To finish, we note that Theorem A opens a window
for the possibility to explore
some further
natural questions which we will discus in the last section of this paper.

\subsection{The organization of the paper}\

\medskip
In Section 2, we
present some
preliminaries.
In Section 3 we provide the two main tools for the proof of Theorem A.
Section 4 is devoted to proving Theorem A.
The proofs for Theorems B and C
are expounded
in Section 5.
In Section 6 we give some applications
of Theorems B and C;
 more specifically we prove there
Theorems D, E and F. Some open questions are discussed
in the final section.

\bigskip
\noindent {\bf Acknowledgement:} We would like to thank
V. Bergelson, N. Frantzikinakis and J.-P. Thouvenot for suggesting some of
the questions we discuss in this work.

\section{Preliminaries}
In this section we give some necessary notions and some known facts which
we will use later.

\subsection{General topological dynamics}\
\medskip

A {\em topological dynamical system} (t.d.s for short) is a triple
$\X=(X, \Gamma, \Pi)$, where $X$ is a compact Hausdorff space, $\Gamma$ is a
Hausdorff topological group and $\Pi: \Gamma\times X\rightarrow X$ is a
continuous map such that $\Pi(e,x)=x$ and
$\Pi(s,\Pi(t,x))=\Pi(st,x)$, where $e$ is the unit of $\Gamma$, $s,t\in \Gamma$ and $x\in X$. We shall fix $\Gamma$ and suppress the
action symbol. Thus for $x\in X$ and $t\in \Gamma$, write $tx$ for $\Pi(t,x)$.


In the paper, we always assume that $X$ is a compact metric space with metric $\rho(\cdot, \cdot)$, and $\Gamma$ is a discrete countable group.
When $\Gamma=\Z$, we will write the t.d.s.
as $(X,T)$ with $T$ being a homeomorphism on $X$. So in this notation $\Gamma=\{T^n: n\in \Z\}$.

\medskip

Let $(X,\Gamma)$ be a t.d.s. and $x\in X$. Then $\O(x,\Gamma)=\{gx: g\in \Gamma\}$ denotes the
{\em orbit} of $x$, which is also denoted by $\Gamma x$. We usually denote the closure of $\O(x,\Gamma )$ by
$\overline{\O}(x,\Gamma)$, or $\overline{ \Gamma x}$.
 Let $A\subseteq X$, the  orbit of $A$ is given by $\O(A,\Gamma)=\{tx: x\in A,t\in \Gamma\}$, and $\overline{\O}(A,\Gamma)= \overline{\O(A,\Gamma)}$.

\medskip

A subset
$A\subseteq X$ is called {\em invariant} (or {$\Gamma$-invariant}) if $g a\subseteq A$ for all
$a\in A$ and $g\in \Gamma$. When $Y\subseteq X$ is a closed and
invariant subset of the system $(X, \Gamma)$, we say that the system
$(Y, \Gamma)$ is a {\em subsystem} of $(X, \Gamma)$. If $(X, \Gamma)$ and $(Y,
\Gamma)$ are two t.d.s., their {\em product system} is the
system $(X \times Y, \Gamma)$, where $g(x, y) = (gx, gy)$ for any $g\in \Gamma$ and $x,y\in X$.

\medskip

A t.d.s. $(X,\Gamma)$ is called {\em minimal} if $X$ contains no proper non-empty
closed invariant subsets. It is easy to verify that a t.d.s. is
minimal if and only if every orbit is dense. In a general system $(X,\Gamma)$ we say that a point $x\in X$ is minimal if $(\overline{\O}(x,\Gamma),\Gamma)$ is minimal.


\medskip

A {\it factor map} $\pi: X\rightarrow Y$ between the t.d.s. $(X,\Gamma)$
and $(Y, \Gamma)$ is a continuous onto map which intertwines the
actions; we say that $(Y, \Gamma)$ is a {\it factor} of $(X,\Gamma)$ and
that $(X, \Gamma)$ is an {\it extension} of $(Y, \Gamma)$.
The systems are said to be {\it isomorphic} if $\pi$ is bijective. Let $\pi: (X,\Gamma)\rightarrow (Y,\Gamma)$ be a factor map. Then
$$R_\pi=\{(x_1,x_2):\pi(x_1)=\pi(x_2)\}$$
is a closed invariant equivalence relation, and $Y=X/ R_\pi$.

\medskip

Let $X, Y$ be compact metric spaces and $T:X \to Y$ be a map.  For $n \geq 2$ let
$T^{(n)}=T\times \ldots \times T \ \text{($n$ times)}: X^n\rightarrow Y^n.$ Thus
we write $(X^n,T^{(n)})$ for the $n$-fold product system $(X\times	\ldots \times X,T\times \ldots \times T)$.
The diagonal of $X^n$ is $$\Delta_n(X)=\{(x,\ldots,x)\in X^n: x\in X\}.$$
When $n=2$ we write	$\Delta(X)=\Delta_2(X)$.

\subsection{Proximal, distal and regionally proximal relations}\
\medskip

Let $(X,\Gamma)$ be a topological dynamical system. Fix $(x,y)\in X^2$. It is a {\it proximal}
pair if $\liminf_{g\in \Gamma} \rho (gx, gy)=0$; it is a {\it
distal} pair if it is not proximal. Denote by ${\bf P}(X,\Gamma)$ the set of proximal pairs of $(X,\Gamma)$. It is also called the {\em proximal relation}. A well known theorem of Auslander-Ellis states that for a  t.d.s. $(X,\Gamma)$, any $x\in X$ is proximal to some minimal point $y$ in $\overline{\O}(x,\Gamma)$.

\medskip

A topological dynamical system $(X,\Gamma)$ is {\it equicontinuous} if for every $\ep>0$ there
exists $\d>0$ such that $\rho (x,y)< \d$ implies $\rho (gx,gy)<\ep$
for every $g\in \Gamma$. It is {\it distal} if ${\bf P}(X,\Gamma)= \D(X)$. Any equicontinuous system is distal.

Let $S_{distal}$ ($S_{eq}$ respectively)  be the smallest closed invariant equivalence relations $S$ on
$X$ for which the factor $X/S$ is a distal (equicontinuous respectively) system. The equivalence relation $S_{distal}$ ($S_{eq}$) is called the {\em {distal}
( equicontinuous) structure relation} of $X$.  It is well known that $S_{distal}$ is the smallest closed invariant equivalence relation on $X$ which includes ${\bf P}(X)$ and $X/S_{distal}$ is the largest distal factor of $X$.

\medskip
In the study of t.d.s., one of the first problems was to characterize $S_{eq}$.
A natural candidate for $S_{eq}$ is
the so-called {\it regionally proximal relation} $\RP(X)$ introduced by Ellis and Gottschalk \cite{EG60}.
Let $(X,\Gamma)$ be a minimal system. The regionally
proximal relation $\RP(X,\Gamma)$ is defined as: $(x,y)\in \RP(X,\Gamma)$ if
for any $\ep>0$ and for any neighborhood $U\times V$
of $(x,y)$ there are $(x',y')\in U\times V$
and $g\in \Gamma$ such that $\rho(gx',gy')<\ep$.
It is well known that $\RP(X,\Gamma)$ is
an invariant closed relation  and this relation
defines the {\em maximal equicontinuous factor} $X_{eq}=X/S_{eq}$ of
$(X,\Gamma)$ (see e.g. \cite[Chapter V]{Vr}).


\medskip

It is a difficult problem to find conditions
under which $\RP(X)$ is an equivalence relation (i.e. $\RP(X)=S_{eq}$). Starting with Veech \cite{V68}, various
authors, including Peterson,
Ellis-Keynes \cite{E-K}, McMahon \cite{Mc78} and Bronstein \cite{Br}, came
up with various sufficient conditions for $\RP(X)$ to be an equivalence relation. What we will use is the following result.


\begin{thm}\label{thm-RP}
Let $(X,\Gamma)$ be a minimal t.d.s., where $\Gamma$ is an amenable group. Then we have the following statements.
\begin{enumerate}
  \item $\RP(X)$ is an invariant closed equivalence relation which induces
the maximal equicontinuous factor $X_{eq}$.
  \item If $(Y,\Gamma)$ is a factor of $(X,\Gamma)$ and we
 let $\pi: X\rightarrow Y$ be a factor map, then $$\pi\times \pi(\RP(X,\Gamma))=\RP(Y,\Gamma).$$
\end{enumerate}
\end{thm}

An extension $\pi:X\ra Y$ is said to be
{\it proximal} if $R_\pi\subset {\bf P}(X)$. The following lemma is well known.

\begin{lem}\label{proximal-prod}
Let $\pi_i: (X_i,\Gamma)\ra (Y_i,\Gamma)$ be proximal extensions, $1\le i\le n.$ Then $$\prod_{i=1}^n\pi_i=\pi_1\times \ldots \times \pi_n: (\prod_{i=1}^nX_i,\Gamma)\rightarrow  (\prod_{i=1}^nY_i,\Gamma)$$
is proximal.
\end{lem}

\subsection{$\RP^{[d]}$ and $\AP^{[d]}$ }\
\medskip

For a t.d.s. $(X,T)$, Host, Kra and Maass \cite{HKM} introduced the following definition.
If ${\bf n} = (n_1,\ldots, n_d)\in \Z^d$ and $\ep\in \{0,1\}^d$, we
define
$${\bf n}\cdot \ep = \sum_{i=1}^d n_i\ep_i .$$

\begin{de}
Let $(X, T)$ be a t.d.s. and let $d\in \N$. The points $x, y \in X$ are
said to be {\em regionally proximal of order $d$} if for any $\d  >
0$, there exist $x', y'\in X$ and a vector ${\bf n} = (n_1,\ldots ,
n_d)\in\Z^d$ such that $\rho (x, x') < \d, \rho (y, y') <\d$, and $$
\rho (T^{{\bf n}\cdot \ep}x', T^{{\bf n}\cdot \ep}y') < \d\
\text{for any $\ep\in \{0,1\}^d\setminus \{\bf 0\}$},$$
where ${\bf 0}=(0,0,\ldots,0)\in \{0,1\}^d$. The set of regionally proximal pairs of
order $d$ is denoted by $\RP^{[d]}$ (or by $\RP^{[d]}(X,T)$ in case of
ambiguity), and is called {\em the regionally proximal relation of
order $d$}. 
\end{de}

Similarly we can define $\RP^{[d]}(X,\Gamma)$ for a system $(X,\Gamma)$ with abelian group $\Gamma$.
We note that $\RP^{[1]}=\RP$. The notion of the regionally proximal relation of
order $d$ was introduced by Host, Kra and Maass in \cite{HKM}.
It is easy to see that $\RP^{[d]}$ is a closed and invariant
relation. Observe that
\begin{equation*}
    {\bf P}(X)\subset  \ldots \subset \RP^{[d+1]}\subset
    \RP^{[d]}\subset \ldots \subset \RP^{[2]}\subset \RP^{[1]}=\RP.
\end{equation*}

\begin{de}
Let $G$ be a group. For $A,B \subset G$, we write $[A,B]$ for the subgroup spanned by $\{[a, b] =aba^{-1}b^{-1}: a \in A, b\in B\}$.
The commutator subgroups $G_j$, $j\ge 1$, are defined inductively by
setting $G_1 = G$ and $G_{j+1} = [G_j ,G]$. Let $d \ge 1$ be an
integer. We say that $G$ is {\em $d$-step nilpotent} if $G_{d+1}$ is
the trivial subgroup.

Let $G$ be a $d$-step nilpotent Lie group and $\Lambda$ be a discrete
cocompact subgroup of $G$. The compact manifold $X = G/\Lambda$ is
called a {\em $d$-step nilmanifold}. The group $G$ acts on $X$ by
left translations and we write this action as $(g, x)\mapsto gx$.
The Haar measure $\mu$ of $X$ is the unique probability measure on
$X$ invariant under this action. Let $\tau\in G$ and $T$ be the
transformation $x\mapsto \tau x$ of $X$. Then $(X, \mu, T)$ is
called a {\em $d$-step nilsystem}. An inverse limit of
$d$-step nilsystems is called a $d$-step {\it pro-nilsystem} or {\em a system of order $d$}.
\end{de}

Host, Kra and Maass \cite{HKM} showed that if a system is minimal and
distal then $\RP^{[d]}$ is an equivalence relation, and a very deep result stating that
$(X/\RP^{[d]},T)$ is the maximal $d$-step pro-nilfactor of the system. 
Shao and Ye \cite{SY} showed that all these results in fact hold for arbitrarily minimal systems
of abelian group actions. See Glasner, Gutman and Ye \cite{GGY18} for similar results regarding general group actions.

The following theorems proved in \cite{HKM} (for minimal distal systems) and
in \cite{SY} (for general minimal systems) tell us conditions under which
the pair $(x,y)$ belongs to $\RP^{[d]}$ and the relation between $\RP^{[d]}$ and
$d$-step pro-nilsystems. We state them for $\Z$-actions and they hold for minimal systems under abelian group actions.

\begin{thm}\label{thm-RP-d}
Let $(X, T)$ be a minimal topological dynamical system and let $d\geq 1$ be an integer. Then
\begin{enumerate}

\item $\RP^{[d]}$ is an equivalence relation.

\item $(X,T)$ is a $d$-step pro-nilsystem if and only if $\RP^{[d]}=\Delta_X$.
\end{enumerate}
\end{thm}

\begin{thm}\label{lift}\label{th3}
Let $\pi: (X,T)\rightarrow (Y,S)$ be a factor map between minimal topological dynamical systems
and let $d\geq 1$ be an integer. Then
\begin{enumerate}
  \item $\pi\times \pi (\RP^{[d]}(X,T))=\RP^{[d]}(Y,S)$.
  \item $(Y,T)$ is a $d$-step pro-nilsystem if and only if $\RP^{[d]}(X,T)\subset R_\pi$.
\end{enumerate}
In particular, the quotient of $(X,T)$ under $\RP^{[d]}(X,T)$ is the
maximal $d$-step pro-nilfactor of $X$ (i.e. the maximal factor which is
$d$-step pro-nilsystem).
\end{thm}

Let $X_d=X/\RP^{[d]}(X,T)$ and $\pi_d: (X,T)\rightarrow (X_d,T)$ be the factor map. The system $X_0$ is the trivial system and the system $X_1$ is the maximal equicontinuous factor $X_{eq}$.

\medskip

The following lemma is an easy consequence from the definition.
\begin{lem}\label{lem-RP}
Let $(X,T)$ be a t.d.s. and $d\in \N$. Then
$$\RP^{[d]}(X,T)=\RP^{[d]}(X,T^n), \quad \forall n\in \N.$$
\end{lem}

\begin{proof}
Let $d,n\in \N$.
It is clear that $\RP^{[d]}(X,T^n)\subset \RP^{[d]}(X,T)$. Now we show that $\RP^{[d]}(X,T)\subset \RP^{[d]}(X,T^n)$.

Let $(x,y)\in \RP^{[d]}(X,T)$ and $\d>0$. There is some $\d'>0$ such that whenever $\rho(x_1,x_2)<\d'$, $\rho(T^{j}x_1,T^jx_2)<\d$ for all $j\in \{1,2,\ldots,dn\}$. Now since $(x,y)\in \RP^{[d]}(X,T)$,
there exist $x', y'\in X$ and a vector ${\bf n'} = (n_1',\ldots ,
n_d')\in\Z^d$ such that $\rho (x, x') < \d', \rho (y, y') <\d'$, and $$
\rho (T^{{\bf n'}\cdot \ep}x', T^{{\bf n'}\cdot \ep}y') < \d'\quad
\text{for any $\ep\in \{0,1\}^d\setminus \{\bf 0\}$}.$$
Thus by the choice of $\d'$, one has that
$$
\rho (T^{{\bf n'}\cdot \ep+j}x', T^{{\bf n'}\cdot \ep+j}y') < \d\quad
\text{for any $\ep\in \{0,1\}^d\setminus \{\bf 0\}$ and $1\le j\le dn$}.$$
Let $n_k=n_k'+j_k\equiv 0 \pmod n$, where $0\le j_k\le n-1$, $k=1,2,\ldots, d$. Note that $${\bf n'}\cdot \ep\le {\bf n}\cdot \ep \le {\bf n'}\cdot \ep +dn.$$ It follows that
$$
\rho (T^{{\bf n}\cdot \ep}x', T^{{\bf n}\cdot \ep}y') < \d\quad
\text{for any $\ep\in \{0,1\}^d\setminus \{\bf 0\}$}.$$
Since $n_k\equiv 0 \pmod n$ for all $k=1,2,\ldots, d$, $(x,y)\in \RP^{[d]}(X,T^n)$.
\end{proof}






Now we give the definition of $\AP^{[d]}$.

\begin{de}\label{arithm}
Let $(X,T)$ be a t.d.s. and $d\in\N$.
We say that $(x,y)\in X\times X$ is a {\em regionally proximal pair of order
$d$ along arithmetic progressions} if for each $\d>0$ there exist
$x',y'\in X$ and $n\in\Z$ such that $\rho(x, x') < \d,
\rho(y, y') <\d$ and $$\rho(T^{in}(x'),
T^{in}(y'))<\d\ \text{for each}\ 1\le i\le d.$$

The set of all such
pairs is denoted by $\AP^{[d]}(X, T)$ or $\AP^{[d]}(X)$ and is called the {\em regionally
proximal relation of order $d$ along arithmetic progressions}.
\end{de}

It follows  easily that $\AP^{[d]}(X,T)\subset \RP^{[d]}(X,T)$ for each $d\in\N$. The following simple observation will be used in the sequel.
Let $(X,T)$ be a t.d.s., $x\in X$, and $d\in \N$. Set $x^{(d)}=(x,x,\ldots,x)\in X^d$.

\begin{lem}\label{sharpnumber}
Let $(X,T)$ be minimal. Then for each $d\ge 3$, $(x^{(d-1)},y)\in N_d(X)$
for some $x,y\in X$ implies that $$(x,y)\in \AP^{[d-2]}(X,T)\subset \RP^{[d-2]}(X,T).$$
Moreover, for $d\ge 3$ and $x\in X$,
$T^{n_i}x\lra x, \ldots, T^{(d-1)n_i}x\lra x, T^{dn_i}x\lra  y$ for some $y$ implies that $(x,y)\in \AP^{[d-1]}(X,T)\subset \RP^{[d-1]}(X,T)$.
\end{lem}

\begin{proof}
Let $x,y\in X$ such that $(x^{(d-1)},y)\in N_d(X)$. There are sequence $\{n_i\}_{i\in \N}, \{m_i\}_{i\in \N}\subset \Z$ such that
$$T^{n_i}x\to x,\ T^{n_i-m_i}x\to x, \ldots, T^{n_i- (d-2)m_i}x\to x,  T^{n_i- (d-1)m_i}x\to y, \quad i\to\infty. $$
Let
$$u_i=T^{n_i-(d-2)m_i}x, \quad v_i= T^{n_i-(d-1)m_i}x, \quad \forall i\in \N.$$
Then $$u_i\to x, \quad v_i \to y, \quad i\to\infty.$$
and for all $1\le j\le d-2$,
$$\rho(T^{jm_i}u_i,T^{jm_i}v_i)=\rho(T^{n_i-(d-2-j)m_i}x,T^{n_i-(d-1-j)m_i}x)
\to \rho(x,x)=0,\quad i\to \infty.$$
By definition $(x,y)\in \AP^{[d-2]}(X,T)$.

\medskip

Now we assume that for $d\ge 3$ and $x\in X$,
$T^{n_i}x\lra x, \ldots, T^{(d-1)n_i}x\lra x, T^{dn_i}x\lra  y$ for some $y$. By similar argument as above, we see that $(x^{(d)},y)\in N_{d+1}(X)$. Thus
$(x,y)\in \AP^{[d-1]}(X,T)\subset \RP^{[d-1]}(X,T)$. The proof is complete.
\end{proof}

\subsection{$\infty$-step nilsystems}\
\medskip

It follows that for any minimal system $(X,T)$,
$\RP^{[\infty]}=\bigcap_{d= 1}^\infty \RP^{[d]}$
is a closed invariant equivalence relation (we write $\RP^{[\infty]}(X,T)$
in case of ambiguity). The following notion
first
appeared in \cite{D-Y}. 

A minimal system $(X, T)$ is an {\em $\infty$-step
pro-nilsystem} or {\em a system of order $\infty$}, if the equivalence
relation $\RP^{[\infty]}$ is trivial, i.e. coincides with the
diagonal.
Similarly, one can show that the quotient of a
minimal system $(X,T)$ under $\RP^{[\infty]}$ is the maximal
$\infty$-step pro-nilfactor of $(X,T)$.

Let $(X,T)$ be a minimal system. It is easy to see that if $(X,T)$
is an inverse limit of minimal nilsystems, then $(X,T)$ is an
$\infty$-step pro-nilsystem. Conversely, if $(X,T)$ is a minimal
$\infty$-step pro-nilsystem, then $\RP^{[\infty]}=\Delta_X$.
So, $(X,T)=\displaystyle
\lim_{\longleftarrow}(X_d,T)_{d\in \N}$ as
$\Delta_X=\RP^{[\infty]}=\bigcap_{d\geq 1}\RP^{[d]}$. In fact
a minimal system is an $\infty$-step pro-nilsystem if and only if it is
an inverse limit of minimal nilsystems \cite{D-Y}.

Since minimal pro-nilsystems are uniquely ergodic, it is easy to see
that minimal $\infty$-step pro-nilsystems are also uniquely ergodic.
\medskip

\subsection{Furstenberg's tower for minimal distal systems}\
\medskip

Let $\pi:(X, \Gamma)\lra (Y,\Gamma)$ be a factor map. We define
the notion of {\it regionally proximal relation relative to $\pi$} (denoted by $\RP_\pi(X)$ or $\RP_\pi$) as follows:
$(x,y)\in \RP_\pi$ if for any neighborhood $U\times V$
of $(x,y)$ and any $\ep>0$ there are $(x',y')\in U\times V$ with $\pi(x')=\pi(y')$
and $g\in \Gamma$ such that $\rho(gx',gy')<\ep$. Thus for $(Y,\Gamma)$ the trivial one point system, we retrieve the regionally proximal relation.
We say that $\pi$ is  an {\it equicontinuous or isometric} extension if for any $\epsilon >0$ there exists $\delta>0$ such that $\pi(x_1)=\pi(x_2)$ and $\rho (x_1,x_2)<\delta$ imply $\rho (gx_1,g x_2)<\epsilon$ for any $g\in \Gamma$.
A factor map $\pi$ is equicontinuous if and only if $\RP_\pi(X)=\Delta(X)$.

\medskip

Furstenberg's structure theorem for minimal distal systems \cite{F63} says that any minimal distal system can be constructed by equicontinuous extensions. Furstenberg showed that if $\pi:X\rightarrow  Y$ is a factor map with $(X, \Gamma)$ minimal and distal, then $\RP_\pi$
is a closed invariant equivalence relation. This gives a structure theorem for minimal distal systems.
That is, for a minimal distal system $(X, \Gamma)$ there is an ordinal $\eta$
(which is countable when $X$ is metrizable) and a family of systems
$\{(F_n,\Gamma)\}_{n\le\eta}$ such that
\begin{enumerate}
  \item[(i)] $F_0$ is a one point trivial system,
  \item[(ii)] for every $n <\eta$ there exists a homomorphism
$\phi_{n+1} :F_{n+1}\to F_{n}$ which is equicontinuous,
  \item[(iii)] for a limit ordinal $\nu\le\eta$ the system $F_\nu$
is the inverse limit of the systems $\{F_\iota\}_{\iota<\nu}$,
\item[(iv)] $F_\eta=X$.
\end{enumerate}
\begin{equation}\label{f-tower}
F_0  \stackrel{\phi_1} \longleftarrow  F_1   \stackrel{\phi_2}\longleftarrow \cdots  \stackrel{\phi_n} \longleftarrow F_n  \stackrel{\phi_{n+1}} \longleftarrow F_{n+1}
 \longleftarrow \cdots \longleftarrow F_\eta=X.
\end{equation}
\eqref{f-tower} is referred to as {\it the Furstenberg tower}.
Note that in \eqref{f-tower} for each $n<\eta$, the system $(F_{n+1},\Gamma)$ is the largest equicontinuous extension of $F_{n}$ within $X$.
That is, if $\psi_n: X\lra F_n$
then $F_{n+1}=F_n/\RP_{\psi_n}$ for $n<\eta$. We call $F_n$  the {\em largest distal factor of order $n$}.

\medskip

We remark that
a simple argument shows the following result.
\begin{prop}
Let $(X,T)$ be a minimal system and $n\in \N$. Then $X_n=X/\RP^{[n]}$ is a factor of $F_n$.
\end{prop}

\begin{proof}
We prove the result by induction on $n\in \N$. When $n=1$, $\RP^{[1]}(X)=\RP(X)$ and $X_1=X/\RP^{[1]}(X)=X/\RP(X)=F_1$.
Now we assume that $X_n$ is a factor of $F_n$ for $n$. Let $\psi_n: X\rightarrow F_n$ and $\pi_n: X\rightarrow X_n$ be the corresponding factor maps.
$$\xymatrix{
                & X\ar[dr]^{\pi_n} \ar[dl]_{\psi_n}             \\
 F_n \ar[rr] & &     X_n        }
$$
It induces the following factor map
$$\phi: F_{n+1}=X/\RP_{\psi_n}(X) \rightarrow X'_{n+1}=X/\RP_{\pi_n}(X).$$
Note $X'_{n+1}$ is the largest equicontinuous extension of $X_{n}$ within $X$. As $X_{n+1}$ is an equicontinuous extension of $X_n$ within $X$. It follows that $X_{n+1}$ is a factor of $X_{n+1}'$. Thus  $X_n$ is a factor of $F_n$.
\end{proof}

An interesting result proved by Qiu and Zhao \cite[Section 6]{QZ-19} is that $F_n=X_n$ for pro-nilsystems.

\begin{lem} \label{qiu}
Let $k,d\in \N$ with $k\le d$ and $(X,T)$ be a minimal $d$-step pro-nilsystem. Then $X_k=X/\RP^{[k]}$
coincides with $F_k$ for $1\le k\le d$.
\end{lem}

\subsection{Some properties of $N_d(X,T)$}\
\medskip

Let $(X,T)$ be a t.d.s., $x\in X$, $A\subseteq X$ and $d\in \N$. Set $x^{(d)}=(x,x,\ldots,x)\in X^d$,
$$\Delta_d(A)=\{x^{(d)}=(x,x,\ldots,x): x\in A\}\subseteq X^d,$$
$$\sigma_d=\sigma(T)=T^{(d)}=T\times \ldots\times T \ (d \ \text{times}),$$ and
$$\tau_d=\tau_d(T)=T\times T^2 \times \ldots \times T^{d}.$$
Note that $\D_d(X)$ is the diagonal of $X^d$.
Let $\G_d=\langle\sigma_d, \tau_d\rangle$, where $\langle\sigma_d, \tau_d\rangle$ denotes the group generated by $\sigma_d$ and $\tau_d$.
Let $\tau_d'=\tau'_d(T)=\id \times T\times \ldots \times T^{d-1}=\id\times \tau_{d-1}.$ Note that $\G_d=\langle \sigma_d, \tau_d\rangle=\langle\sigma_d, \tau'_d \rangle$, which will used frequently in the paper.


\medskip

Let $X, Y$ be sets, and let $\pi : X\rightarrow Y$ be a	map. A subset $L$ of $X$ is called
{\em $\pi$-saturated} if $$\{x\in L: \pi^{-1}(\pi(x))\subseteq L\}=L,$$ i.e. $L=\pi^{-1}(\pi(L))$.

\begin{de}\cite{G94}
Let $\pi: (X,T)\rightarrow (Y,T)$ be a factor map of topological systems and $d\in \N$.
$(Y,T)$ is said to be a {\em $d$-step topological characteristic factor (along $\tau_d$) }
or {\em topological characteristic factor of order $d$} if there exists a dense $G_\d$	set $X_0$ of $X$ such
that for each $x\in X_0$ the orbit closure $$L_x=\overline{\O}(x^{(d)}, \tau_d)$$ is $\pi^{(d)}=\pi\times \ldots \times \pi$
($d$ times) saturated. That is, $(x_1,x_2,\ldots, x_d)\in L_x$ if and only if $(x_1',x_2',\ldots, x_d')\in L_x$
whenever for all $i\in \{1,2,\ldots, d\}$, $\pi(x_i)=\pi(x_i')$.
\end{de}

\begin{thm}\cite{G94}\label{thm-Glasner-distalcase}
If $(X,T)$ is a distal minimal system and $d\ge 2$, then its largest distal factor of order $d-1$ $F_{d-1}$ is its topological characteristic factor of order $d$.
\end{thm}

\begin{thm}\label{thm-nil-case}
Let $(X,T)$ be a $d$-step nilsystem for some $d\in \N$. Then for each $1\le i\le d-1$,
$X_i$ is
a
$(i+1)$-step topological characteristic factor of $X$.
\end{thm}
\begin{proof} This following from Theorem \ref{thm-Glasner-distalcase} and Lemma \ref{qiu}.
\end{proof}

Let $(X,T)$ be a t.d.s. and $d\in N$. Let
$$N_d(X,T)=N_d(X)=\overline{\O}(\D_d(X), \tau_d).$$
If $(X,T)$ is transitive and $x\in X$ is a transitive point, then
$N_d({X})=\overline{\O}(x^{(d)},
\langle \sigma_d,\tau_d \rangle).$

\medskip

We want to emphasize that $N_d(X,T)$ also plays
a key role in the study of the pointwise convergence of (\ref{m-ave}) for ergodic distal m.p.t.
In particular, it is shown
in
\cite{HSY-point} that each ergodic m.p.t. admits a uinquely ergodic minimal model
for which
$Z_d=X_d$ for $d\in\N$, where $Z_d$ is the measurable pro-nilfactor defined in \cite{HK05}.

\medskip


The next theorem is fundamental for the analysis carried throughout our work
(for a short enveloping semigroup proof see \cite[Proposition 1.55]{G-book}):

\begin{thm}[Glasner]\cite{G94}\label{thm-Glasner}
Let $(X,T)$ be a minimal t.d.s. and $d\in\N$. Then the system $(N_d(X), \langle \sigma_d, \tau_d\rangle)$
is minimal and the $\tau_d$-minimal points are dense in $N_d(X)$.
\end{thm}

By the same proof of Theorem \ref{thm-Glasner}, we  have

\begin{thm}\label{thm-Glasner-1}
Let $(X,T)$ be a minimal t.d.s. and $a_1,a_2,\ldots, a_d$ be distinct numbers of $ \Z$, where $d\in\N$. Let
$$\tau=T^{a_1}\times T^{a_2}\times \ldots \times T^{a_d}.$$
Then $\big(\overline{\O}(\Delta_d(X),\tau), \langle\sigma_d, \tau \rangle \big)$ is minimal and the $\tau$-minimal points are dense in $\overline{\O}(\Delta_d(X),\tau)$.
\end{thm}

The following two lemmas follow from \cite{G94}.

\begin{lem}\label{lem-intt}
Let $(X,T)$ be a minimal system and $d\in \N$. Let $U\subset X$ be a non-empty open subset and let $U^{(d)}=\Delta_d(U)=\{x^{(d)}: x\in U\}$. Then
$$\intt _{N_{d}(X)}\overline{\O}(U^{(d)},\tau_d)\neq \emptyset.$$
\end{lem}

\begin{proof}
Let $U\subset X$ be a non-empty open subset.
Since $(X,T)$ is minimal, there is some $K\in \N$ such that
$X=\bigcup_{k=1}^{K}T^{-k}U$. It follows that
 $\Delta_d(X)=\bigcup_{k=1}^K (T^{(d)})^{-k} U^{(d)}$. Thus
$$N_d(X)= \overline{\O}(\Delta_d(X),\tau_d)=\overline{\O}(\bigcup_{k=1}^K (T^{(d)})^{-k} U^{(d)} ,\tau_d)= \bigcup_{k=1}^K (T^{(d)})^{-k} \overline{\O}( U^{(d)} ,\tau_d).$$
Therefore $\intt _{N_{d}(X)}\overline{\O}(U^{(d)},\tau_d)\neq \emptyset.$
\end{proof}

\begin{lem}\label{eli-thm} Let $d\in\N$, $(X,T)$ be a distal minimal t.d.s. and $\psi_d:X\lra F_d$ be the factor map to its largest distal factor of order $d$.
Then there is a dense $G_\delta$ set $\Omega$ of $X$ such that if $x\in \Omega$, then $\overline{Orb}(x^{(d)},\tau_d)$ is $\psi_{d}^{(d)}$ is saturated and
$(\psi_{d}^{(d+2)})^{-1}N_{d+2}(F_d)=N_{d+2}(X)$.
\end{lem}

\begin{proof}
The first part is from Theorem \ref{thm-Glasner-distalcase}, and the sencond part follows from Lemma \ref{lem-saturated} in the sequel.
\end{proof}

\subsection{Open extensions}\
\medskip

Let $(X,T)$ and $(Y,S)$ be t.d.s. and let $\pi: X \to Y$ be a factor map.
One says that:
\begin{enumerate}
  \item $\pi$ is an {\it open} extension if it is open as a map
  \item $\pi$ is an {\it almost one to one} extension  if there
exists a dense $G_\d$ set $\Omega\subseteq X$ such that
$\pi^{-1}(\{\pi(x)\})=\{x\}$ for any $x\in \Omega$;
\end{enumerate}


We will often use the following
construction which is due originally to Veech
(see \cite[Theorem 3.1]{Veech})

\begin{thm}\label{O}
Given a factor map $\pi:X\rightarrow Y$ between minimal systems
$(X,T)$ and $(Y,S)$, there exists a commutative diagram of factor
maps (called {\em O-diagram})
\[
\begin{CD}
X @<{\sigma^*}<< X^*\\
@V{\pi}VV      @VV{\pi^*}V\\
Y @<{\tau^*}<< Y^*
\end{CD}
\]
such that

\noindent (a) $\sigma^*$ and $\tau^*$ are almost one to one
extensions;

\noindent (b) $\pi^*$ is an open extension;

\noindent (c) $X^*$ is the unique minimal set in $R_{\pi
\tau^*}=\{(x,y)\in X\times Y^*: \pi(x)=\tau^* (y)\}$ and
$\sigma^*$ and $\pi^*$ are the restrictions to $X^*$ of the
projections of $X\times Y^*$ onto $X$ and $Y^*$ respectively.
\end{thm}

We note that this diagram is canonical. In particular, when the map $\pi$ is open, we have $X^*=X$.

\subsection{Some subsets of $\Z$}\
\medskip

A subset $S$ of $\mathbb{Z}$ is {\it syndetic} if it has a bounded
gap, i.e. there is $N \in \mathbb{N}$ such that $\{i, i+1, \ldots,
i+N\} \cap S \neq \emptyset$ for every $i \in \mathbb{Z}$. A subset $S\subset \Z$
is {\it thick} if it contains arbitrarily long runs of positive
integers, i.e., for every $n \in \mathbb{N}$ there exists some $a_n
\in \mathbb{Z}$ such that $\{a_n, a_n+1, \ldots, a_n+n\} \subset S$.

A subset $S$ of $\mathbb{Z}$ is {\it piecewise syndetic} if it is the intersection of a syndetic set with a thick set;
and it is {\it thickly syndetic} if for each $n\in\N$ there is a syndetic
subset $\{w^n_1,w^n_2, \ldots\}$ of $S$ such that
$\{w^n_i,w^n_i+1,\ldots,w^n_i+n\}\subset S$ for each $i\in\N$.
Denote by $\F_{ts}$ the family of all thickly syndetic sets. It is clear that if $F_1,F_2\in \F_{ts}$ so is $F_1\cap F_2$. That is, $\F_{ts}$ is a filter.

\section{The main tools used in proving Theorem A}

In this section we will introduce the main tools to show Theorem A. We start from the Saturation theorem.

\subsection{A Saturation Theorem}\
\medskip

In \cite{G94}
Glasner  proved an auxiliary theorem in order to prove saturation with respect to the Furstenberg's tower.
By simplifying the assumptions of this theorem, we find that it applies to a more general setup.
We will discuss
this theorem, and other lemmas related to
saturation properties in the sequel.

\begin{lem}\cite[Lemma 2.1.]{G90}\label{Lem-Glasner}
Let $\phi: X\rightarrow Y$ be an open map of compact metric spaces. Let $\V=\{V\subseteq X: V$ open and $\phi(V)=Y\}$.
Then there exists a countable subset $\{V_i\}_{i=1}^\infty$ of $\V$ such that every element of $\V$ contains some $V_i$.
\end{lem}

\begin{proof}
First we show that for each $V\in\V$, there exists a closed subset $L_V\subset V$ with $\phi(L_V)=Y$. Let $V\in \V$. If $V^c=\emptyset$, then set $L_V=V$. If $V^c\not =\emptyset$, then for $\ep>0$ denote $$V_\ep=\{x\in V: d(x, V^c)\ge \ep\}.$$
If for every $n\in \N$, $\phi(V_{1/n})\neq Y$, then there exits $y_n\in Y\setminus \phi(V_{1/n})$. Let $\lim_{n\to\infty}y_{n}=y$ without loss of generality. By assumption there exits $x\in V$ with $\phi(x)=y$. Let $\d=d(x,V^c)$. By our assumption, $\d>0$. Since $\phi$ is open we can find $x_n\in X$ with $\phi(x_n)=y_n$ such that $\lim _{n\to\infty} x_n=x$. But then $x_n$ is eventually in $V_{\d/2}$ and $y_n\in \phi(V_{\d/2})$, a contradiction. Thus we proved that for every $V\in \V$ there exists a closed subset $L_V\subseteq V$ with $\phi(L_V)=Y$.

Let $\U=\{U_i\}_{i=1}^\infty$ be a countable basis for open sets on $X$. Then for each $V\in \V$, one can find a finite subset $\{U_{i_1},\ldots, U_{i_k}\}$ covers $L_V$ and satisfies $\bigcup_{j=1}^kU_{i_j}\subseteq V$. Thus
$$\V_0=\Big\{V=\bigcup_{j=1}^kU_{i_j}: U_{i_j}\in \U\ \text{and }\ \phi(V)=Y\Big\}$$
is the required countable collection of subsets.
\end{proof}

We derive a useful lemma from \cite[Lemma 3.3]{G94}. For completeness, we include a proof.

\begin{thm}[Saturation theorem]\label{Thm-Glasner's-Lemma}
Let $I$ be some index set and for each $\zeta\in I$ let $\sigma_\zeta: (X_\zeta,\Gamma)\rightarrow (Z_\zeta, \Gamma)$ be an extension of t.d.s.( not necessarily minimal systems), where $\Gamma$ is a discrete countable group. Let
$$(X,\Gamma)=(\prod_{\zeta\in I} X_\zeta, \Gamma), \ \ (Z,\Gamma)=(\prod_{\zeta\in I} Z_\zeta, \Gamma)$$
and let $\sigma: X\rightarrow Z$ be the product homomorphisms.
Let $N_Z$ be a non-empty closed $\Gamma$-invariant subset of $Z$, and let $N_X=\sigma^{-1}(N_Z)$.

Let $Q$ be a closed subset of $N_X$. Suppose that
\begin{enumerate}
 \item $\sigma$ is open, i.e. $\sigma_\zeta$ is open for each $\zeta \in I$;
 \item For every non-empty relatively open set $U\subseteq {N_X}$, one has that $$\overline{\O}(U,\Gamma)=\sigma^{-1}\big(\sigma (\overline{\O}(U,\Gamma))\big);$$
   \item $\overline{\O}(Q, \Gamma)=N_X$;
   \item for every non-empty relatively open subset $U$ of $Q$, ${\rm int}_{N_X}(\overline{\O}(U, \Gamma))\neq\emptyset.$
\end{enumerate}
Then there exists a dense $G_\delta$ subset $\Omega$ of $Q$ such that ${\bf x}\in \Omega$ implies $\overline{\O}({\bf x},\Gamma)$ is $\sigma$-saturated.
\end{thm}

\begin{proof}
Let $S$ be a non-empty closed subset of $N_Z$ for which $\cl(\intt_{N_Z}(S))=S$, and let
\begin{equation*}
  Q_S=\{x\in Q: \O(x,\Gamma)\cap \sigma^{-1}[\intt_{N_Z}(S)]\neq \emptyset\}.
\end{equation*}
Then by (3), $Q_S$ is a non-empty open subset of $Q$.

\medskip

\noindent{\bf Claim:}\ {\em
There exists a dense $G_\d$ subset $\Omega_S$ of $Q_S$ such that for each $x\in \Omega_S$, there exits $z\in S$ with $$\sigma^{-1}(z)\subseteq L_x=\overline{\O}(x,\Gamma).$$}

\noindent {\em Proof of Claim.}\
Let
\begin{equation*}
  \V=\{V\subseteq N_X: V\cap \sigma^{-1}[S] \ \text{is relatively open in $\sigma^{-1}[S]$ and } \ \sigma[V]\supseteq S\}.
\end{equation*}
Since $\sigma(U\cap \sigma^{-1}(S))=\sigma(U)\cap S$ for each open set $U$ of $N_X$,  by openness of $\sigma$ it follows that $\sigma|_{\sigma^{-1}[S]}: \sigma^{-1}(S)\rightarrow S$ is open. Therefore by Lemma \ref{Lem-Glasner} there exists a countable sub-collection $\{V_k\}_{k=1}^\infty$ of $\V$ such that each $V\in \V$ contains an element of $\{V_k\}_{k=1}^\infty$.

For each $k\in \N$, $V_k\cap \sigma^{-1}(S)$ is relatively open in $\sigma^{-1}(S)$ and $\sigma[V_k]\supseteq S$, so by the fact that $\intt_{N_Z} (S)\neq \emptyset$, we conclude that $\intt_{N_X}(V_k)\neq \emptyset$. By condition (2) and $\cl(\intt_{N_Z}(S))=S$ we get
\begin{equation}\label{g1}
  \overline{\O}(\intt_{N_X} (V_k),\Gamma) = \sigma^{-1}\big[\sigma[\overline{\O} (\intt_{N_X} (V_k),\Gamma)]\big]\supseteq \overline{\O}(\sigma^{-1}[S],\Gamma).
\end{equation}

For each $k\in \N$, let
\begin{equation*}
  \Lambda_k=\{x\in Q_S: \O(x,\Gamma)\cap V_k\neq \emptyset\}.
\end{equation*}
Let $U$ be a non-empty open subset of $Q_S$, then $U$ is a relatively open subset of $Q$ (since $Q_S$ is an open subset of $Q$) and by (4) $\intt_{N_X}\overline{\O}(U,\Gamma)\neq \emptyset$.
From this fact we verify that
$$\overline{\O}(U,\Gamma)\cap {\O}(V_k,\Gamma)\neq \emptyset.$$
For each $x\in U\subset Q_S$, we have
$$\O(x,\Gamma)\cap \sigma^{-1}[\intt_{N_Z}(S)]\neq \emptyset.$$
Thus $x\in \O(\sigma^{-1}[\intt_{N_Z}(S)],\Gamma),$ i.e.
$$U\subset \O(\sigma^{-1}[\intt_{N_Z}(S)],\Gamma).$$
By \eqref{g1},
$$\overline{\O}(U,\Gamma)\subset
\overline{\O}(\sigma^{-1}[\intt_{N_Z}(S)],\Gamma)\subset  \overline{\O}(\intt_{N_X} (V_k),\Gamma). $$
By $\intt_{N_X}\overline{\O}(U,\Gamma)\neq \emptyset$, we have that
$$\overline{\O}(U,\Gamma)\cap {\O}(V_k,\Gamma)\neq \emptyset.$$
Hence $\O(U,\Gamma)\cap V_k\neq \emptyset$, and it follows that
$\Lambda_k$ is an open dense subset of $Q_S$. The set $$\Omega_S=\bigcap_{k=1}^\infty \Lambda_k$$ is therefore a dense $G_\d$ subset of $Q_S$. In particular, for $S=N_Z$, $\Omega_{N_Z}$ is a dense $G_\d$ subset of $Q$.

Now for $x\in \Omega_S$ we show that there exits $z\in S$ with $\sigma^{-1}(z)\subseteq L_x=\overline{\O}(x,\Gamma)$. Let $V=N_X\setminus L_x$. Then $V$ is an open $\Gamma$-invariant subset of $N_X$ and, if $\sigma[V]\supseteq S$, then $V\in \V$ and for some $k$, $V_k\subseteq V$. Since then, however, $x\in \overline{\O}(V_k,\Gamma)\subseteq V$, we get a contradiction. Thus there exists $z\in S$ with $\sigma^{-1}(z)\cap V=\emptyset$, i.e. $\sigma ^{-1}(z)\subseteq L_x$.

The proof of Claim is complete.
\hfill $\square$
\medskip

Let $\{ B_0=N_Z, B_1,B_2,\ldots\}$ be a basis for the topology of $N_Z$. Define $S_j=\overline{B_j}$ and, inductively, we define dense $G_\d$ subset $\Omega_j$ of $Q$ as follows. Let $\Omega_0=\Omega_{S_0}=\Omega_{N_Z}$. We put
\begin{equation*}
  \Omega_{j+1}=(\Omega_j\cap \Omega_{S_{j+1}})\cup \Big(\Omega_j\cap (\overline{Q_{S_{j+1}}})^c\Big).
\end{equation*}
Notice that this is a disjoint union. Finally let $\Omega=\bigcap_{j=0}^\infty \Omega_j$.

Let $x_0\in \Omega$ and denote $L=\overline{\O}(x_0,\Gamma)$. Put
\begin{equation*}
  L_\sigma=\{x\in L:\sigma^{-1}(\sigma(x))\subseteq L\}.
\end{equation*}
Since $\sigma$ is open, $L_\sigma$ is closed and clearly $L_\sigma$ is $\sigma$-saturated. If  $L_\sigma=L$, we are done. Otherwise, let $z_0\in \sigma(L)\setminus \sigma(L_\sigma)$. Since $N_Z\setminus \sigma(L_\sigma)$ is open and $\{B_i\}_{i=0}^\infty$ is a base of $N_Z$, there is some $j$ such that $z_0\in B_j\subseteq S_j\subseteq N_Z\setminus \sigma(L_\sigma)$.

Now
\begin{equation*}
  x_0\in \Omega\subset \Omega_j=(\Omega_{j-1}\cap \Omega_{S_{j}})\cup\Big(\Omega_{j-1}\cap (\overline{Q_{S_{j}}})^c\Big).
\end{equation*}
Since $z_0\in \sigma(L)\cap B_j$, we have that
$$\O(x_0,\Gamma)\cap \sigma^{-1}(B_j)\neq \emptyset,$$  and hence  $x_0\in Q_{S_{j}}$. Therefore $x_0\in \Omega_{j-1}\cap \Omega_{S_j}\subset \Omega_{S_j}$. By Claim there exits $z\in S_j$ with $\sigma^{-1}(z)\subseteq L$, whence $\sigma^{-1}(z)\subseteq L_\sigma$; this contradicts with $S_j\subseteq N_Z\setminus \sigma(L_\sigma)$.
To sum up, for all ${x_0}\in \Omega$ we have that $\overline{\O}({x_0},\Gamma)$ is $\sigma$-saturated. The proof is completed.
\end{proof}

The following lemma was implicitly proved in \cite{G94}, we give a proof for completeness.
\begin{lem}\label{lem-saturated}
Let $\pi: (X,T)\rightarrow (Y,T)$ be an open extension of minimal systems and $d\in \N$. If $Y$ is
a
$d$-step topological characteristic factor of $X$, then $N_{d+1}(X)$ is $\pi^{(d+1)}$-saturated, i.e.
$$(\pi^{(d+1)})^{-1}(N_{d+1}(Y))=N_{d+1}(X).$$
\end{lem}

\begin{proof}
By the hypothesis, there is a dense $G_\delta$ subset $\Omega_d$ of $X$ such that for any $x\in \Omega_d$
	\begin{equation*}
	\overline{\O}({x}^{(d)},\tau_{d})
	\end{equation*}
is $\pi^{(d)}$-saturated. It is clear that
	\begin{equation*}
N_{d+1}(X)\subset (\pi^{(d+1)})^{-1}(N_{d+1}(Y)).
	\end{equation*}
Now we prove the converse.  Let ${ y}\in Y$
and ${x}	\in \pi^{-1}(y)$. Since $\Omega_d$ is dense, choose $\{{ x}_i\}_{i\in\N}\subseteq \Omega_d$ such that ${x}_i\rightarrow {x}, i\to\infty$. Let ${y}_i=\pi({x}_i)$, then ${y}_i\rightarrow {y}, i\to\infty.$ Since $\{ x_i\}_{i\in\N}\subseteq \Omega_d$, for each $i\in \N$
	\begin{equation*}
\overline{\O}((x_i)^{(d)},\tau_{d})=(\pi^{(d)})^{-1} \left(\overline{\O}({y_i}^{(d)},\tau_{d})\right).
	\end{equation*}
It follows that (recall here $\tau'_{d+1}=\id \times \tau_d$)
\begin{equation*}
\begin{split}	
& \quad \{{ x}_i\}\times(\pi^{(d)})^{-1}\left(\overline{\O}({ y_i}^{(d)},\tau_{d})\right)\\
	&=\{{x}_i\}\times \overline{\O}({x}_i^{(d)},\tau_{d}) \\
	&=\overline{\O}({ x}_i^{(d+1)},{\rm id} \times \tau_{d})=\overline{\O}({ x}_i^{(d+1)},\tau'_{d+1})\\
	&\subseteq \overline{\O}(\Delta_{d+1}(X), \tau'_{d+1})=\overline{\O}(\Delta_{d+1}(X),\tau_{d+1}).
\end{split}		
\end{equation*}
In particular,
	\begin{equation*}
	\{{ x}_i\}\times(\pi^{(d)})^{-1}\left({ y_i}^{(d)}\right)\subseteq \overline{\O}(\Delta_{d+1}(X),\tau_{d+1}).
	\end{equation*}
Note that $\pi^{-1}$ is continuous as $\pi$ is open, and we have
	\begin{equation*}
	\begin{split}
	&\quad \{{ x}\}\times(\pi^{(d)})^{-1}\left({ y}^{(d)}\right)=\lim_{i\to\infty}\{{ x}_i\}\times(\pi^{(d)})^{-1}\left({ y_i}^{(d)}\right)  \\
	&\subseteq \overline{\O}(\Delta_{d+1}(X),\tau_{d+1}).
	\end{split}
	\end{equation*}
That is,
	\begin{equation*}
	(\pi^{(d+1)})^{-1}\left({ y}^{(d+1)}\right)= \pi^{-1}({ y})\times(\pi^{(d)})^{-1}\left({ y}^{(d)}\right)  \subseteq \overline{\O}(\Delta_{d+1}(X),\tau_{d+1}).
	\end{equation*}
	Since ${y}\in Y$ is arbitrary, we have
	\begin{equation*}
	(\pi^{(d+1)})^{-1}(\Delta_{d+1}(Y))
	\subseteq \overline{\O}(\Delta_{d+1}(X),\tau_{d+1}).
	\end{equation*}
	By the continuity of $\pi^{-1}$, we have
	\begin{equation*}
	\begin{split}
	&\quad (\pi^{(d+1)})^{-1}(N_{d+1}(Y))= (\pi^{(d+1)})^{-1}(\overline{\O}(\Delta_{d+1}(Y)
	 ,\tau_{d+1})) \\	&=\overline{\O}((\pi^{(d+1)})^{-1}(\Delta_{d+1}(Y))
	 ,\tau_{d+1})\\
	&\subseteq \overline{\O}(\Delta_{d+1}(X),\tau_{d+1})=N_{d+1}(X).
	\end{split}
	\end{equation*}
The proof is completed.
\end{proof}

\medskip
The following lemma is easy to verify, by the definitions.

\begin{lem}\label{lem-saturated-basic}
Let $X,Y,Z$ be compact metric spaces.
Let $\pi: X\rightarrow Y, \phi: X\rightarrow Z, \psi: Z\rightarrow Y$ be continuous surjective maps such that $\pi= \psi\circ \phi$
$$\xymatrix@R=0.5cm{
  X \ar[dd]_{\pi} \ar[dr]^{\phi}             \\
                & Z \ar[dl]_{\psi}         \\
  Y                 }
$$
\begin{enumerate}
  \item If $A\subset X$ is $\pi$ saturated, then $A$ is $\phi$ saturated.
  \item If $A\subset X$ is $\phi$ saturated and $\phi(A)$ is $\psi$ saturated, then $A$ is $\pi$ saturated.
\end{enumerate}
\end{lem}

\subsection{The connection of $\RP^{[d]}$ with recurrence sets}\
\medskip

We now turn to the second tool
we use in the proof of
Theorem A.
We need the notions of Poincar\'e and
Birkhoff recurrence sets of higher order, see \cite{F77, FLW}. To define them we
appeal to
the MERT and MTRT theorems stated in the introduction.






\begin{de}
Let $d\in \N$.
\begin{enumerate}
\item We say that $S \subset \Z$ is a set of {\em $d$-recurrence } if
for every measure preserving system $(X,\X,\mu,T)$ and for every
$A\in \X$ with $\mu (A)
> 0$, there exists $n \in S$  such that
$\mu(A\cap T^{-n}A\cap \ldots \cap T^{-dn}A)>0.$


\item We say that $S\subset \Z$ is a set of {\em $d$-topological
recurrence} if for every minimal t.d.s. $(X, T)$ and for every
nonempty open subset $U$ of $X$, there exists $n\in S$ such that
$U\cap T^{-n}U\cap \ldots \cap T^{-dn}U\neq \emptyset.$

\item We say that $S\subset \Z$ is a Nil$_d$ Bohr$_0$-set, if there are a $d$-step
nilsystem $(X,T)$, $x\in X$ and a neighbourhood $U$ of $x$ such that $S\supset N(x,U)$.
\end{enumerate}
\end{de}

Let $\F_{Poi_d}$ (resp. $\F_{Bir_d}$, $\F_d$) be
the family consisting of all sets of $d$-recurrence (resp. sets of
$d$-topological recurrence, the sets of Nil$_d$ Bohr$_0$-set). It is obvious by the definitions above
that $\F_{Poi_d}\subset \F_{Bir_d}$.

\medskip

Let $(X,T)$ be a t.d.s., $x\in X$ and $U\subset X$. Set
$$N_T(x,U)=\{n\in \Z: T^nx\in U\}.$$

The following result plays an important role in the proof of Theorem A.
\begin{thm}\cite[Theorem 7.2.7]{HSY16}\label{several}
Let $(X,T)$ be a minimal t.d.s., $d\in\N$ and $x,y\in X$. Then the
following statements are equivalent:

\begin{enumerate}

\item $(x,y)\in \RP^{[d]}(X,T)$.

\item
$N_T(x,U)\in \F_{Poi_d}$ for each neighborhood $U$ of $y$.

\item $N_T(x,U)\in \F_{Bir_d}$ for each
neighborhood $U$ of $y$.

\item $N_T(x,U)\in \F_d^*$, i.e. $N(x,U)\cap A\not=\emptyset$ for each Nil$_d$ Bohr$_0$-set $A$.
\end{enumerate}
\end{thm}

Since this theorem is one of the main ingredients in the proof of Theorem A, let us
say
some
words as to
why it is true.

\medskip

To do this, we first note that by mainly using the result in \cite{BHK-10} we have the following fact: for any minimal
t.d.s. $(X,T)$, $d\in\N$ and any nonempty open set $U$ of $X$,
$$M_d=\{n\in \Z:U\cap T^{-n}U\cap \ldots \cap T^{-dn}U\not=\emptyset\}$$
is almost a Nil$_d$ Bohr$_0$-set, in the sense that the difference of $M_d$ with a Nil$_d$ Bohr$_0$-set has zero density.

Another fact we need is: if $d\in\N$ and $A\subset \Z$ is a Nil$_d$ Bohr$_0$-set, then there exists a minimal $d$-step nilsystem $(X,T)$ and a non-empty
open subset $V$ of $X$ with
$$A\supset \{n\in \Z:V\cap T^{-n}V\cap \ldots\cap T^{-dn}V\not=\emptyset\}.$$
The proof of the latter fact is very involved: it is proved through generalized polynomials of degree $d$ and a very complicated computation.
As a corollary of the latter fact
we get immediately that
$$ \F_{Poi_d}\subset \F_{Bir_d}\subset \F_d^*.$$ So, it remains to
to show
that (1) implies (2) and
that
(4) implies (1) in Theorem \ref{several}.

\medskip

In the proof of the implication
 (4) to (1) we need to use the Ellis semigroup theory and some non-trivial discussions.
 One useful consequence of
the proof is that $N_T(x,U)\cap A$ (stated in (4) of Theorem \ref{several}) is a syndetic set, which implies that it has positive upper density.

To prove (1) implies (2) we mainly use the first fact and the consequence we obtain from the proof (4) implies (1).
Once we have
this, then the proof follows by all
the
elements mentioned above.

\section{Proof of Theorem A}

With the preparation in Section 3, we are in
a
position to show Theorem A.
First we
prove
a key lemma.


\subsection{Proof of Theorem A assuming a  key lemma}\
\medskip

We will give a very useful lemma in this subsection, which allows us to reduce problems related to general minimal systems
to the nilfactors.
Recall that $X_d=X/\RP^{[d]}$ for $d\in \N$ and $\pi_d: (X,T)\rightarrow (X_d,T)$ is the corresponding factor map.
For $j<i$, let $$\pi_{i,j}: (X_i,T)\rightarrow (X_j,T).$$ Then we have
$$\xymatrix{
                & X\ar[dr]^{\pi_j} \ar[dl]_{\pi_i}             \\
 X_i \ar[rr]^{\pi_{i,j}} & &     X_j        }
$$


The following lemma plays a key role in the proof of Theorem A, whose proof will be given in later subsections
since it is very long. We remark that perhaps the number $d'$ picked in the lemma is not sharp, but it is convenient
for us to get all the information we need.

\begin{lem}\label{lem-Key}
Let $\pi:(X,T)\rightarrow (Y,T)$ be an extension of minimal systems, $d\in \N$. Let $\pi$ be open and $X_{d'}$ be a factor of $Y$ with $d'\ge 2d!(d-1)!$.
$$
\xymatrix{
                &         X \ar[d]^{\pi} \ar[dl]_{\pi_{d'}}    \\
  X_{d'} & Y   \ar[l]_{\phi}          }
$$
Then for every non-empty relatively open subset $O$ of ${N_d(X)}$, one has $$\overline{\O}(O,\tau_d)
 =(\pi^{(d)})^{-1}\big(\pi^{(d)} (\overline{\O}(O,\tau_d))\big).$$
\end{lem}

Now we prove Theorem A assuming Lemma \ref{lem-Key}. In fact, Theorem A follows from the following theorem and the $O$-diagram construction.

\begin{thm}\label{key}
Let $\pi:(X,T)\rightarrow (Y,T)$ be an extension of minimal systems.
If $\pi$ is open and $X_{\infty}$ is a factor of $Y$, then $Y$ is a $d$-step topological
characteristic factor of $X$ for all $d\in \N$.
$$
\xymatrix{
                &         X \ar[d]^{\pi} \ar[dl]_{\pi_\infty}    \\
  X_{\infty} & Y   \ar[l]_{\phi}          }
$$
That is, for all $d\in \N$ there exists a dense $G_\d$ subset $\Omega_d$ of $X$ such that for each $x\in \Omega_d$ the orbit closure $L_x=\overline{\O}(x^{(d)}, \tau_{d})$ is $\pi^{(d)}$-saturated.

\end{thm}


\begin{proof}[Proof of Theorem \ref{key} assuming Lemma \ref{lem-Key}]
We prove the theorem by induction on $d$. Case $d=1$ is trivial.

\medskip
\noindent \textbf{Case $d=2$.} \quad In this case $N_2(X)=X\times X$, $N_2(Y)= Y\times Y$. Let $\pi: (X,T)\rightarrow (Y,T)$. Then also we have $\pi: (X,T^2)\rightarrow (Y,T^2)$. Let
$$\pi^{(2)}: (X\times X,T\times T^2) \rightarrow (Y\times Y,T\times T^2).$$

We need to verify the following:
\begin{enumerate}
 \item[$(1)_2$] $\pi^{(2)}$ is open;
 \item[$(2)_2$] For every non-empty relatively open subset $O$ of ${N_2(X)}$, one has that $$\overline{\O}(O,T\times T^2)
 =(\pi^{(2)})^{-1}\big(\pi^{(2)} (\overline{\O}(O,T\times T^2))\big);$$

   \item[$(3)_2$] $\overline{\O}(\Delta(X), T\times T^2)=N_2(X)=(\pi^{(2)})^{-1}(N_2(Y))$;
   \item[$(4)_2$] for every  non-empty open subset $U$ of $X$, ${\rm int}_{N_2(X)}(\overline{\O}(U^{(2)}, T\times T^2))\neq\emptyset.$
\end{enumerate}

$(1)_2$ is from our assumption, and $(3)_2$ is clear. By Lemma \ref{lem-intt}, we have $(4)_2$.  $(2)_2$ follows from Lemma \ref{lem-Key}.

\medskip

Then by Saturation Theorem (Theorem \ref{Thm-Glasner's-Lemma}) there exists a dense $G_\delta$ subset $\Omega_2$ of
$X$ such that ${ x}\in \Omega_2$ implies $\overline{\O}({x^{(2)}}, T\times T^2)$ is $\pi^{(2)}$-saturated.
That is, $Y$ is a $2$-step topological characteristic factor of $X$.

\medskip
\noindent \textbf{Case $d+1$.}\ We assume that the result holds for $d\ge 2$, and we show $d+1$.
We will verify the conditions of the  Saturation Theorem. That is, we will verify the following conditions:
\begin{enumerate}
 \item[$(1)_{d+1}$] $\pi^{(d+1)}$ is open;
 \medskip
 \item[$(2)_{d+1}$] For every non-empty relatively  open set $O$ of ${N_{d+1}(X)}$, one has that $$\overline{\O}(O,\tau_{d+1})=(\pi^{(d+1)})^{-1}\Big(\pi^{(d+1)} (\overline{\O}(O,\tau_{d+1}))\Big);$$

   \item[$(3)_{d+1}$] $\overline{\O}(\Delta_{d+1}(X), \tau_{d+1})=N_{d+1}(X)=(\pi^{(d+1)})^{-1}(N_{d+1}(Y))$;
       \medskip
   \item[$(4)_{d+1}$] for every non-empty open subset $U$ of $X$, ${\rm int}_{N_{d+1}(X)}(\overline{\O}(U^{(d+1)}, \tau_{d+1})\neq\emptyset.$
\end{enumerate}
\medskip

Condition $(1)_{d+1}$ follows from our assumption that $\pi$ is open. $(2)_{d+1}$ follows from Lemma \ref{lem-Key}. By inductive hypothesis on $d$, $Y$ is a $d$-step topological characteristic factor of $X$, then by Lemma \ref{lem-saturated}, $N_{d+1}(X)$ is $\pi^{(d+1)}$-saturated, i.e.
$$(\pi^{(d+1)})^{-1}(N_{d+1}(Y))=N_{d+1}(X).$$
Hence we have $(3)_{d+1}$. By Lemma \ref{lem-intt}, we have $(4)_{d+1}$.

So by Theorem \ref{Thm-Glasner's-Lemma} there exists a dense $G_\delta$ subset $\Omega_{d+1}$ of
$X$ such that ${x}\in \Omega_{d+1}$ implies $\overline{\O}({x^{(d+1)}}, \tau_{d+1})$ is $\pi^{(d+1)}$-saturated.
That is, $Y$ is
a
$d+1$-step topological characteristic factor of $X$. The proof is completed.
\end{proof}

We are ready to show Theorem A assuming Lemma \ref{lem-Key}.

\begin{proof}[Proof of Theorem A] It follows Theorem \ref{key} and the $O$-diagram.
\end{proof}


As a corollary of Theorem A, we have

\begin{thm}\label{main-distal}
Let $(X,T)$ be a minimal system which is an open extension of its maximal distal factor,
and $d\in \N$. Then $X_d$ is
a
$(d+1)$-step topological characteristic factor of $X$.
\end{thm}
\begin{proof}
Let $(X,T)$ is a minimal system which is an open extension of its maximal distal factor and $d\in \N$. We want to show that $X_d$ is
a
$d+1$-step topological characteristic factor of $X$. Since $X_\infty$ is a factor of the maximal distal factor, by our assumption $\pi_\infty: X\rightarrow X_\infty$ is open. By the proof of Theorem \ref{key}, there is some $\widetilde{d}$ such that $X_{\widetilde{d}}$ is a $d+1$-step topological characteristic factor of $X$.
$$\xymatrix@R=0.5cm{
  X \ar[dd]_{\pi_1} \ar[dr]^{\pi_{\widetilde{d}}}             \\
                & X_{\widetilde{d}} \ar[dl]^{\pi_{\widetilde{d},d}}         \\
  X_d                 }
$$
By Theorem \ref{thm-nil-case}, $X_d$ is
a
$d+1$-step topological characteristic factor of $X_{\widetilde{d}}$.
By Lemma \ref{lem-saturated-basic}, $X_d$ is
a
$d+1$-step topological characteristic factor of $X$.
\end{proof}

We now proceed to the proof of Lemma \ref{lem-Key}.

\subsection{Cases $d=1,\ d=2$ for Lemma \ref{lem-Key}}\
\medskip

In this subsection, to make the idea clear, we show the cases $d=1$ and $d=2$ of Lemma~ \ref{lem-Key}, and we show the general case after that.

Let $\rho$ be the metric of $X$ and $\rho_d$ the metric of $X^d$ defined by
$$\rho_d({\bf x},{\bf y})=\max_{1\le j\le d} \rho(x_j,y_j),$$
where ${\bf x}=(x_1,x_2,\ldots, x_d), {\bf y}=(y_1,y_2,\ldots,y_d) \in X^d$.

\medskip
\noindent \textbf{Case $d=1$.} \quad In this case , we take $d'=1$:
$$
\xymatrix{
                &         X \ar[d]^{\pi} \ar[dl]_{\pi_{1}}    \\
  X_{1} & Y   \ar[l]_{\phi}          }
$$
It is easy to see that
$N_2(X)=X\times X$, $N_2(Y)= Y\times Y$. Let $\pi: (X,T)\rightarrow (Y,T)$. Note that $R_\pi\subset \RP^{[1]}(X)$.  Let
$$\pi^{(2)}: (X\times X,T\times T^2) \rightarrow (Y\times Y,T\times T^2).$$

For any non-empty open set $O\subset N_2(X)$, we need to show that
$$N=:\overline{\O}(O,T\times T^2) =(\pi^{(2)})^{-1}\big(\pi^{(2)} (\overline{\O}(O,T\times T^2))\big).$$

To show this,
we first show the following claim:

\bigskip
\noindent{\bf Claim:} {\em Let $(x_1,x_2)\in O$. Then
$\pi^{-1}(\pi(x_1))\times \{x_2\}\subset N$, and $\{x_1\}\times \pi^{-1}(\pi(x_2))\subset N$}

\medskip

Let $(x_1,x_2)\in O$. Let $\ep>0$ such that $B_\ep(x_1)\times B_\ep(x_2)\subset O$. Let $z_1\in X$ such that $(x_1,z_1)\in R_\pi\subset \RP^{[1]}(X)$.

Since $(x_1,z_1)\in R_\pi\subset \RP^{[1]}(X)$, we know by Theorem \ref{several} that $N(x_1,B_\ep(z_1))\in  \F_{Bir_1}$. By the definition of $\F_{Bir_1}$, there is $n\in\Z$ such that
$$T^n(x_1)\in B_\ep(z_1)\ \text{and}\ A=B_\ep(x_2)\cap T^{-2n} B_\ep(x_2)\not=\emptyset.$$
Pick $x_2'\in A$, and we have $T^{n}x_1 \in B_\ep(z_1)$ and $T^{2n}x_2'\in B_\ep(x_2).$ It is clear that
$$(x_1,x_2')\in B_\ep(x_1)\times B_\ep(x_2)\subset O.$$
Thus $$\rho_2((z_1,x_2), \overline{\O}(O, T\times T^2))\le \rho_2((z_1,x_2), \overline{\O}((x_1,x_2'),T\times T^2))<\ep$$

Since $\ep$ is arbitrary, we know that $(z_1,x_2)\in N=\overline{\O}(O, T\times T^2))$. This implies that
 $$\pi^{-1}\pi(x_1)\times \{x_2\}\subset N.$$

Now let $z_2\in X$ with $(x_2,z_2)\in R_\pi\subset \RP^{[1]}(X)$.
Since $(x_2,z_2)\in R_\pi\subset \RP^{[1]}(X,T)=\RP^{[1]}(X,T^2)$, by Theorem \ref{several} $$N_{T^2}(x_2,B_\ep(z_2))\in \F_{Bir_1}.$$
Thus by the definition of $ \F_{Bir_1}$, there is some $n$ such that
$$T^{2n}x_2 \in B_{\ep}(z_2),$$
and
$$B_\ep(x_1)\cap T^{-n}B_{\ep}(x_1)\neq \emptyset.$$
Let $x_1' \in B_\ep(x_1)\cap T^{-n}B_{\ep}(x_1)$. Then
$$T^n x_1'\in B_\ep (x_1), \quad T^{2n}x_2 \in B_\ep (z_2).$$
Thus $\rho_2((T\times T^2)^n(x_1',x_2), (x_1,z_2))<\ep.$
Note that $(x_1',x_2)\subset B_\ep(x_1)\times B_\ep(x_2)\subset O\subset N.$
It follows
$$\rho_2((x_1,z_2), \overline{\O}(O, T\times T^2) )\le \rho_2((x_1,z_2), \overline{\O}((x_1',x_2),T\times T^2))<\ep.$$
Since $\ep$ is arbitrary, $(x_1,x_2')\in N=\overline{\O}(O, T\times T^2).$
Thus
$$\{x_1\}\times \pi^{-1}(\pi(x_2))\subset N$$
The proof of Claim is completed. \hfill $\square$

\medskip

Let $$(z_1,z_2)\in  N=\overline{\O}(O,\tau_2).$$
We will show that
$$(z_1',z_2')\in N,$$
where $\pi(z_i)=\pi(z_i'), i=1,2.$

First we show that $(z_1',z_2)\in N$.
Since $(z_1,z_2)\in  N=\overline{\O}(O,\tau_2)$, there are some sequences $\{(x_1^i,x_2^i)\}_{i\in \N}\subset  O$ and $\{n_i\}_{i\in \N}\subset \Z$ such that
$$\tau_2^{n_i}(x_1^i,x_2^i)\to (z_1,z_2), \ i\to\infty.$$
By Claim and $(x_1^i,x_2^i)\in O$,
$$\pi^{-1}(\pi(x_1^i))\times \{x_2^i\}\subset N.$$
Since $\pi$ is open, it follows that
\begin{equation*}
  \begin{split}
   &\quad  \tau_2^{n_i}\Big (\pi^{-1}(\pi(x_1^i))\times \{x_2^i\}\Big )
     =\pi^{-1}(\pi(T^{n_i }x_1^i))\times \{T^{2n_i}x_2^i\} \\
    & \xrightarrow{i\to\infty} \pi^{-1}(\pi(z_1))\times \{z_2\}
     \subset \overline{\O}(N,\tau_2)= N.
   \end{split}
\end{equation*}
Thus we have
$(z_1',z_2)\in N.$
Similarly, we have
$(z_1',z_2')\in N.$

\medskip

Thus we finished the case $d=1$.

\medskip

\medskip
\noindent \textbf{Case $d=2$.} \
For every non-empty relatively open set $O\subseteq {N_3(X)}$, we want to show  that $$N=\overline{\O}(O,\tau_3)=(\pi^{(3)})^{-1}\Big(\pi^{(3)} (\overline{\O}(O,\tau_3))\Big),$$
where $\tau_3=T\times T^2\times T^3$.

In this case , we take $d'=4$:
$$
\xymatrix{
                &         X \ar[d]^{\pi} \ar[dl]_{\pi_{4}}    \\
  X_{4} & Y   \ar[l]_{\phi}          }
$$

Let

$$\pi^{(3)}: (X^3, \tau_3) \rightarrow (Y^3, \tau_3).$$

$$
\xymatrix{
                &         (X^3,\tau_3) \ar[d]^{\pi^{(3)}} \ar[dl]_{\pi^{(3)}_{4}}    \\
 ( X_{4}^{3},\tau_3) & (Y^{3},\tau_3)   \ar[l]_{\phi^{(3)}}          }
$$

Let $\tau_{3,1}=\id\times T\times T^2$, $\tau_{3,2}=T\times \id\times T^{-1}$ and $\tau_{3,3}=T^2\times T\times \id$.
It is easy to see that
$$\langle\tau_3, T^{(3)}\rangle=\langle\tau_{3,1,}, T^{(3)}\rangle=\langle\tau_{3,2}, T^{(3)}\rangle=\langle\tau_{3,3}, T^{(3)}\rangle.$$


\medskip
\noindent{\bf Step 1}.  {\em Let $(x_1,x_2,x_3)\in O$. If $y\in X$ with  $\pi(x_1)=\pi(y)$, then $(y,x_2, x_3)\in N=\overline{\O}(O, \tau_3)$. }
\medskip

Since $(x_1,x_2,x_3) \in O$ and $O$ is a non-empty relatively open subset of ${N_3(X)}$, there is some $\d>0$ such that
$$B_{\d}((x_1,x_2,x_3))\subset O.$$

By Theorem \ref{thm-Glasner-1}, $(N_3(X), \langle \id\times T\times T^2, T^{(3)}\rangle)$ is minimal and the $\id\times T\times T^2$-minimal points in $N_3(X)$ are dense in $N_3(X)$.

Let $\ep>0$ with $\ep<\frac{\d}{2}$.
Choose an $\id\times T\times T^2$-minimal point $(z_1,z_2, z_3)\in N_{3}(X)$ such that
$$\rho_3((z_1,z_2, z_3),(x_1,x_2,x_3) )<\ep<\d/2.$$
By the openness of $\pi$, we may assume that there is some $y_1 \in X$ such that $$\rho(y_1,y)< {\ep}\ \text{and}\ \pi(z_1)=\pi(y_1).$$

Let
$$A=\overline{\O}((z_1,z_2,z_3), \id\times T\times T^2).$$
Since $(z_1,z_2,z_3)$ is $\id\times T\times T^2$-minimal, $(A, \id\times T\times T^2)$ is a minimal system.
Note that by the definition of
$A$, for all $(w_1,w_2, w_3)\in A$, one has that $w_1=z_1$.

Let
$$U=A\cap \Big(B_{\ep}(z_1)\times B_{\ep}(z_2)\times B_{\ep}(z_3)\Big)$$
be a non-empty open subset of $A$.

By the assumption $(z_1, y_1)\in R_\pi\subset \RP^{[4]}(X,T)=\RP^{[4]}(X,T^2)$, we have $$N_{T^2}(z_1,B_\ep(y_1))\in \F_{Bir_4}.$$
Thus, according to Theorem \ref{several}
there is some $n\in \Z$ such that
$$T^{2n}z_1\in B_\ep(y_1),$$
and
$$B=U\cap (\id\times T\times T^2)^{-n}U\cap (\id\times T\times T^2)^{-2n}U\cap (\id \times T\times T^2)^{-3n}U\cap(\id\times T\times T^2)^{-4n}U\neq \emptyset.$$

Let $(z_1, y_2,y_3)\in B$. Then from $(z_1, y_2,y_3)\in (\id\times T\times T^2)^{-4n}U$, we get
$T^{4n}y_2\in B_\ep(z_2)$, and from $(z_1, y_2,y_3)\in (\id \times T\times T^2)^{-3n}U$ we get $T^{6n}y_3\in B_\ep(z_3)$
(this explains why we need to use $\RP^{[4]}$ instead of $\RP^{[2]}$ for our method).
Hence, we have
\begin{equation}\label{}
 T^{2n} z_1\in B_\ep(y_1),\ T^{4n}y_2\in B_\ep(z_2),\  T^{6n}y_3\in B_\ep(z_3).
\end{equation}
Thus
$$\rho_3\Big(\tau_3^{2n}(z_1,y_2,y_3),(y_1,z_2,z_3)\Big)<\ep,$$
and
$$\rho_3\Big((y_1,z_2,z_{3}),\overline{\O}((z_1,y_2,y_3),\tau_3)\Big)< \ep.$$
Since $\rho_3((z_1,z_2,z_3),(x_1,x_2,x_3))<\ep$ and $\rho(y_1,y)< \ep$,
we have
$$\rho_3\Big((y_1,z_{2},z_3), (y,x_{2}, x_{3}) \Big)<\ep .$$
Thus
\begin{equation}\label{list-4-2}
  \rho_3\Big((y, x_{2}, x_3), \overline{\O}((z_1,y_2,y_3),\tau_3) \Big)<\ep+\ep=2\ep.
\end{equation}
Since $(z_1,y_2,y_3)\in U$, $\rho_3((z_1,y_2,y_3),(z_1,z_2,z_3))<\ep$, it follows that
$$\rho_3((z_1,y_2,y_3), (x_1,x_2,x_3)) < \rho_3((z_1,y_2,y_3), (z_1,z_2,z_3)) + \rho_d((z_1,z_2,z_3), (x_1,x_2,x_3)) <\ep+\ep<\d.$$
Hence $$(z_1,y_2,y_3) \in B_\d((x_1,x_2,x_3))\subset O.$$
So $\overline{\O}((z_1,y_2,y_3),\tau_3)\subset \overline{\O}(O , \tau_3) $, and we have
$$\rho_3\Big((y, x_2, x_3), \overline{\O}(O , \tau_3) \Big)<2\ep,$$
by (\ref{list-4-2}).
As $\ep$ is arbitrary, we have that indeed
$$(y, x_2, x_3) \in \overline{\O}(O,\tau_3).$$

\medskip
\noindent{\bf Step 2}.  {\em Let $(x_1,x_2,x_3)\in O$. If $y\in X$ with  $\pi(x_2)=\pi(y)$, then $(x_1,y, x_3)\in N$. }
\medskip

Since $(x_1,x_2,x_3) \in O$ and $O$ is a non-empty relatively open subset of ${N_3(X)}$, there is some $\d>0$ such that
$$B_{\d}((x_1,x_2,x_3))\subset O.$$

By Theorem \ref{thm-Glasner-1}, $(N_3(X), \langle T\times \id \times T^{-1}, T^{(3)}\rangle)$ is minimal and the $T\times \id \times T^{-1}$-minimal points in $N_3(X)$ are dense in $N_3(X)$.

Let $\ep>0$ with $\ep<\frac{\d}{2}$.
Choose an $T\times \id \times T^{-1}$-minimal point $(z_1,z_2, z_3)\in N_{3}(X)$ such that
$$\rho_3((z_1,z_2, z_3),(x_1,x_2,x_3) )<\ep<\d/2.$$
By the openness of $\pi$, we may assume that there is some $y_2 \in X$ such that $$\rho(y_2,y)< {\ep}\ \text{and}\ \pi(z_2)=\pi(y_2).$$

Let
$$A=\overline{\O}((z_1,z_2,z_3), T\times \id \times T^{-1}).$$
Since $(z_1,z_2,z_3)$ is $T\times \id \times T^{-1}$-minimal, $(A, T\times \id \times T^{-1})$ is a minimal system.
Note that by the definition of $T\times \id \times T^{-1}$, for all $(w_1,w_2, w_3)\in A$, one has that $w_2=z_2$.

Let
$$U=A\cap \Big(B_{\ep}(z_1)\times B_{\ep}(z_2)\times B_{\ep}(z_3)\Big)$$
be a non-empty open subset of $A$.

By Theorem \ref{several} and the assumption $(z_2, y_2)\in R_\pi\subset \RP^{[4]}(X,T)=\RP^{[4]}(X,T^2)$, we have $$N_{T^2}(z_2,B_\ep(y_2))\in \F_{Bir_4}.$$
Thus by the definition of $\F_{Bir_4}$ there is some $n\in \Z$ such that
$$T^{2n}z_2\in B_\ep(y_2),$$
and
$$B=U\cap (T\times \id \times T^{-1})^{-n}U\cap (T\times \id \times T^{-1})^{-2n}U\cap (T\times \id \times T^{-1})^{-3n}U\cap(T\times \id \times T^{-1})^{-4n}U\neq \emptyset.$$
Let $(y_1, z_2, y_3)\in B$. Then we have
\begin{equation}\label{}
 T^{n} y_1\in B_\ep(z_1),\ T^{2n}z_2\in B_\ep(y_2),\  T^{3n}y_3\in B_\ep(z_3).
\end{equation}
Thus
$$\rho_3\Big(\tau_3^{n}(y_1,z_2,y_3),(z_1,y_2,z_3)\Big)<\ep,$$
and
$$\rho_3\Big((z_1,y_2,z_{3}),\overline{\O}((y_1,z_2,y_3),\tau_3)\Big)< \ep.$$
Since $\rho_3((z_1,z_2,z_3),(x_1,x_2,x_3))<\ep$ and $\rho(y_2,y)< \ep$,
we have
$$\rho_3\Big((z_1,y_{2},z_3), (x_1,y , x_{3}) \Big)<\ep .$$
Thus
\begin{equation}\label{list-4-4}
  \rho_3\Big((x_1, y, x_3), \overline{\O}((y_1,z_2,y_3),\tau_3) \Big)<\ep+\ep=2\ep.
\end{equation}
Since $(y_1,z_2,y_3)\in U$, $\rho_3((y_1,z_2,y_3),(z_1,z_2,z_3))<\ep$, it follows that
$$\rho_3((y_1,z_2,y_3), (x_1,x_2,x_3)) < \rho_3((y_1,z_2,y_3), (z_1,z_2,z_3)) + \rho_d((z_1,z_2,z_3), (x_1,x_2,x_3)) <\ep+\ep<\d.$$
Hence $$(y_1,z_2,y_3) \in B_\d((x_1,x_2,x_3))\subset O.$$
So $\overline{\O}((y_1,z_2,y_3),\tau_3)\subset \overline{\O}(O , \tau_3) $, and we have
$$\rho_3\Big((x_1,y, x_3), \overline{\O}(O , \tau_3) \Big)<2\ep,$$
by (\ref{list-4-4}).
As $\ep$ is arbitrary, we have
$$( x_1,y, x_3) \in \overline{\O}(O,\tau_3).$$

\medskip
\noindent{\bf Step 3}.  {\em Let $(x_1,x_2,x_3)\in O$. If $y\in X$ with  $\pi(x_3)=\pi(y)$, then $(x_1,x_2, y)\in N$. }
\medskip

Since $(x_1,x_2,x_3) \in O$ and $O$ is a non-empty relatively open subset of ${N_3(X)}$, there is some $\d>0$ such that
$$B_{\d}((x_1,x_2,x_3))\subset O.$$

By Theorem \ref{thm-Glasner-1}, $(N_3(X), \langle T^2\times  T\times \id , T^{(3)}\rangle)$ is minimal and the $T^2\times  T\times \id$-minimal points in $N_3(X)$ are dense in $N_3(X)$.

Let $\ep>0$ with $\ep<\frac{\d}{2}$.
Choose an $T^2\times  T\times \id$-minimal point $(z_1,z_2, z_3)\in N_{3}(X)$ such that
$$\rho_3((z_1,z_2, z_3),(x_1,x_2,x_3) )<\ep<\d/2.$$
By the openness of $\pi$, we may assume that there is some $y_3 \in X$ such that $$\rho(y_3,y)< {\ep}\ \text{and}\ \pi(z_3)=\pi(y_3).$$

Let
$$A=\overline{\O}((z_1,z_2,z_3), T^2\times  T\times \id).$$
Since $(z_1,z_2,z_3)$ is $T^2\times  T\times \id$-minimal, $(A, T^2\times  T\times \id)$ is a minimal system.
Note that by the definition of $T^2\times  T\times \id$, for all $(w_1,w_2, w_3)\in A$, one has that $w_3=z_3$.

Let
$$U=A\cap \Big(B_{\ep}(z_1)\times B_{\ep}(z_2)\times B_{\ep}(z_3)\Big)$$
be a non-empty open subset of $A$.

By Theorem \ref{several} and the assumption $(z_3, y_3)\in R_\pi\subset \RP^{[4]}(X,T)=\RP^{[4]}(X,T^6)$, we have $$N_{T^6}(z_3,B_\ep(y_3))\in \F_{Bir_4}.$$
Thus be the definition of $\F_{Bir_4}$ there is some $n\in \Z$ such that
$$T^{6n}z_3\in B_\ep(y_3),$$
and
$$B=U\cap (T^2\times  T\times \id)^{-n}U\cap (T^2\times  T\times \id)^{-2n}U\cap (T^2\times  T\times \id)^{-3n}U\cap(T^2\times  T\times \id)^{-4n}U\neq \emptyset.$$
Let $(y_1, y_2, z_3)\in B$. Then we have
\begin{equation}\label{}
 T^{2n} y_1\in B_\ep(z_1),\ T^{4n}y_2\in B_\ep(z_2),\  T^{6n}z_3\in B_\ep(y_3).
\end{equation}
Thus
$$\rho_3\Big(\tau_3^{2n}(y_1,y_2,z_3),(z_1,z_2,y_3)\Big)<\ep,$$
and
$$\rho_3\Big((z_1,z_2,y_{3}),\overline{\O}((y_1,y_2,z_3),\tau_3)\Big)< \ep.$$
Since $\rho_3((z_1,z_2,z_3),(x_1,x_2,x_3))<\ep$ and $\rho(y_3,y)< \ep$,
we have
$$\rho_3\Big((z_1,z_{2},y_3), (x_1, x_{2},y) \Big)<\ep .$$
Thus
\begin{equation}\label{list-4-6}
  \rho_3\Big((x_1, x_2, y), \overline{\O}((y_1,y_2,z_3),\tau_3) \Big)<\ep+\ep=2\ep.
\end{equation}
Since $(y_1,y_2,z_3)\in U$, $\rho_3((y_1,y_2,z_3),(z_1,z_2,z_3))<\ep$, it follows that
$$\rho_3((y_1,y_2,z_3), (x_1,x_2,x_3)) < \rho_3((y_1,y_2,z_3), (z_1,z_2,z_3)) + \rho_d((z_1,z_2,z_3), (x_1,x_2,x_3)) <\ep+\ep<\d.$$
Hence $$(y_1,y_2,z_3) \in B_\d((x_1,x_2,x_3))\subset O.$$
So $\overline{\O}((y_1,y_2,z_3),\tau_3)\subset \overline{\O}(O , \tau_3) $, and we have that
$$\rho_3\Big((x_1,x_2, y), \overline{\O}(O , \tau_3) \Big)<2\ep,$$ by (\ref{list-4-6}).
As $\ep$ is arbitrary, we have
$$( x_1,x_2, y) \in \overline{\O}(O,\tau_3).$$

\medskip
\noindent{\bf Step 4}.  {\em $N=(\pi^{(3)})^{-1}(\pi^{(3)}N)$. }
\medskip

Let $$(z_1,z_2,z_3)\in  N=\overline{\O}(O,\tau_3).$$
We will show that
$$(z_1',z_2',z_3')\in N,$$
where $\pi(z_i)=\pi(z_i'), i=1,2,3.$

First we show that $(z_1',z_2,z_3)\in N$.
Since $(z_1,z_2,z_3)\in  N=\overline{\O}(O,\tau_3)$, there are some sequences $\{(x_1^i,x_2^i,x_3^i)\}_{i\in \N}\subset  O$ and $\{n_i\}_{i\in \N}\subset \Z$ such that
$$\tau_3^{n_i}(x_1^i,x_2^i,x_3^i)\to (z_1,z_2,z_3), \ i\to\infty.$$
By step 1 and $(x_1^i,x_2^i,x_3^i)\in O$,
$$\pi^{-1}(\pi(x_1^i))\times \{x_2^i\}\times \{x^i_3\}\subset N.$$
Since $\pi$ is open, it follows that
\begin{equation*}
  \begin{split}
   &\quad  \tau_3^{n_i}\Big (\pi^{-1}(\pi(x_1^i))\times \{x_2^i\}\times \{x^i_3\}\Big ) \\
    & =\pi^{-1}(\pi(T^{n_i }x_1^i))\times \{T^{2n_i}x_2^i\}\times \{T^{3n_i} x^i_3\} \\
    & \xrightarrow{i\to\infty} \pi^{-1}(\pi(z_1))\times \{z_2\}\times \{z_3\} \\
    & \subset \overline{\O}(N,\tau_3)= N.
   \end{split}
\end{equation*}
Thus we have
$$(z_1',z_2,z_3)\in N.$$
Similarly, using Step 2 we have
$$(z_1',z_2',z_3)\in N.$$
And then using Step 3 we have
$$(z_1',z_2',z_3')\in N.$$

\medskip

Thus we have finished the proof for the case $d=2$.

\subsection{Proof of Lemma \ref{lem-Key} in the general case}
\begin{proof}[Proof of Lemma \ref{lem-Key}]
Let $\rho$ be the metric of $X$ and $\rho_d$ the metric of $X^d$ defined by
$$\rho_d({\bf x},{\bf y})=\max_{1\le j\le d} \rho(x_j,y_j),$$
where ${\bf x}=(x_1,x_2,\ldots, x_d), {\bf y}=(y_1,y_2,\ldots,y_d)\in X^d$.

\medskip

Let $O$ be a non-empty relatively open subset of ${N_d(X)}$.
First we show the following claim.

\medskip

\noindent \textbf{Claim:} \  {\em Let ${\bf x}=(x_1,x_2,\ldots,x_d)\in O$ and $j\in \{1,2,\ldots, d\}$. Then for each $y\in X$ with $\pi(y)=\pi(x_j)$, one has that
$$(x_1,\ldots, x_{j-1}, y,x_{j+1},\ldots, x_d) \in \overline{\O}(O,\tau_d).$$
}

\noindent {\em Proof of the Claim.} \
Since ${\bf x}\in O$ and $O$ is a non-empty relatively open subset of ${N_d(X)}$, there is some $\d>0$ such that
$$B_{\d}({\bf x})\subset O.$$

Let
\begin{equation}\label{}
S_j=\tau_d^{-1} (T^{(d)})^{j}=T^{j-1}\times T^{j-2}\times \ldots \times T\times \id \times T^{-1}\times \ldots \times T^{-(d-j)}, 1\le j\le d.
\end{equation}
Note that
$$\langle\tau_d, T^{(d)}\rangle=\langle S_j, T^{(d)}\rangle.$$
By Theorem \ref{thm-Glasner-1}, $(N_d(X), \langle S_j, T^{(d)}\rangle)$ is minimal and the $S_j$-minimal points in $N_d(X)$ are dense in $N_d(X)$.

Let $\ep>0$ with $\ep<\frac{\d}{2}$.
Choose an $S_j$-minimal point ${\bf z}=(z_1,z_2,\ldots, z_d)\in N_{d}(X)$ such that
$$\rho_d({\bf z},{\bf x})<\ep<\d/2.$$
By the openness of $\pi$, we may assume that there is some $y_j \in X$ such that $$\rho(y_j,y)< {\ep}\ \text{and}\ \pi(z_j)=\pi(y_j).$$

Let
$$A=\overline{\O}({\bf z}, S_j).$$
Since ${\bf z}$ is $S_j$-minimal, $(A, S_j)$ is a minimal system.
Note that by the definition of $S_j$, for all ${\bf w}=(w_1,w_2,\ldots, w_d)\in A$, one has that $w_j=z_j$.

Let
$$U=A\cap \Big(B_{\ep}(z_1)\times B_{\ep}(z_2)\times \ldots \times B_{\ep}(z_d)\Big)$$
be a non-empty open subset of $A$.

Since $X_{d'}$ is a factor of $Y$, we have that $(z_j,y_j)\in \RP^{[d']}(X,T)$, where $d'= 2d!(d-1)!$. By Lemma \ref{lem-RP}, $\RP^{[d']}(X,T)=\RP^{[d']}(X,T^{j(d-1)!})$. Thus by Theorem \ref{several}, $N_{T^{j(d-1)!}}(z_j, V)\in \F_{Bir_{d'}}$ for each
neighborhood $V$ of $y_j$. Together with the definition of $\F_{Bir_{d'}}$, there is some $n\in \Z$ such that
\begin{equation}\label{h1}
  T^{j(d-1)!n}z_j\in B_{\ep }(y_j),
\end{equation}
and
\begin{equation}\label{h2}
  U\cap S^{-n}_j U\cap S_j^{-2n} U\cap \ldots \cap S^{-d'n}_jU\neq \emptyset.
\end{equation}
Let
\begin{equation*}
  {\bf w'}\in U\cap S^{-n}_j U\cap S_j^{-2n} U\cap \ldots \cap S^{-d'n}_jU\neq \emptyset ,
\end{equation*}
and
\begin{equation*}
 {\bf w}=(w_1,w_2,\ldots,w_d)= S_j^{\frac{d'}{2}n}{\bf w'}\in U.
\end{equation*}
Then by \eqref{h2},
$$S_j^{in}{\bf w}=S_j^{(\frac{d'}{2}+i)n}{\bf w'}\in U, \quad \text{for all $i$ with }\  -\frac{d'}{2}\le i\le \frac{d'}{2}.$$
That is, for all $i$ with $-\frac{d'}{2}\le i\le \frac{d'}{2}$, we have
\begin{equation}\label{h3}
  \begin{split}
     T^{(j-1)in }w_1 & \in B_{\ep}(z_1),\\
      T^{(j-2)in }w_2 & \in B_{\ep}(z_2),\\ & \ldots, \\
      T^{in }w_{j-1}& \in B_{\ep}(z_{j-1}), \\
        w_j& =z_j, \\
       T^{-in }w_{j+1}& \in B_{\ep}(z_{j+1}),\\ & \ldots,\\ T^{-(d-j)in }w_d& \in B_{\ep}(z_d).
   \end{split}
\end{equation}
Since $d'= 2d!(d-1)!$, by \eqref{h3} and \eqref{h1} we have
\begin{equation*}
  \begin{split}
     T^{(d-1)!n }w_1 & \in B_{\ep}(z_1),\\
      T^{(d-1)!2n }w_2 & \in B_{\ep}(z_2),\\ & \ldots, \\
      T^{(d-1)!(j-1)n }w_{j-1}& \in B_{\ep}(z_{j-1}), \\
       T^{(d-1)!jn} w_j& \in B_{\ep}(y_j), \\
       T^{(d-1)!(j+1)n }w_{j+1}& \in B_{\ep}(z_{j+1}),\\ & \ldots,\\ T^{(d-1)!d n }w_d & \in B_{\ep}(z_d).
   \end{split}
\end{equation*}
It follows that
$$\rho_d\Big( \tau_d^{(d-1)!n} {\bf w}, (z_1,\ldots,z_{j-1},y_j,z_{j+1},\ldots, z_d) \Big)< \ep,$$
and hence
$$\rho_d\Big((z_1,\ldots,z_{j-1},y_j,z_{j+1},\ldots, z_d),\overline{\O}({\bf w},\tau_d)\Big)< \ep.$$
Since $\rho_d({\bf z},{\bf x})<\ep$ and $\rho(y_j,y)< \ep$,
we have that
$$\rho_d\Big((z_1,\ldots,z_{j-1},y_j,z_{j+1},\ldots, z_d), (x_1,\ldots, x_{j-1}, y,x_{j+1},\ldots, x_d) \Big)<\ep .$$
Thus
\begin{equation}\label{h4}
  \rho_d\Big((x_1,\ldots, x_{j-1}, y,x_{j+1},\ldots, x_d), \overline{\O}({\bf w},\tau_d) \Big)<\ep+\ep=2\ep.
\end{equation}
Since ${\bf w}\in U$, $\rho_d({\bf w},{\bf z})<\ep$, it follows that
$$\rho_d({\bf w}, {\bf x}) < \rho_d({\bf w}, {\bf z}) + \rho_d({\bf z}, {\bf x}) <\ep+\ep<\d.$$
Hence $${\bf w}\in B_\d({\bf x})\subset O.$$

By \eqref{h4}, we have
$$\rho_d\Big((x_1,\ldots, x_{j-1}, y,x_{j+1},\ldots, x_d), \overline{\O}(O , \tau_d) \Big)<2\ep.$$
As $\ep$ is arbitrary, we have
$$(x_1,\ldots, x_{j-1}, y,x_{j+1},\ldots, x_d) \in \overline{\O}(O,\tau_d).$$
The proof of Claim is completed. \hfill $\square$

\medskip

Now we will use the Claim to show that the orbit closure $L=\overline{\O}(O, \tau_{d})$ is $\pi^{(d)}$-saturated.

For $j\in \{1,2,\ldots, d\}$, let
$${\bf z}=(z_1,z_2,\ldots, z_d)\in L=\overline{\O}(O,\tau_d).$$
We show that $${\bf z'}=(z_1,z_2,\ldots, z_{j-1}, z_j',z_{j+1},\ldots, z_d)\in L,$$
where $\pi(z_j')=\pi(z_j)$.

Since ${\bf z}=(z_1,z_2,\ldots, z_d)\in L=\overline{\O}(O,\tau_d),$
there are some sequences $\{{\bf x^i}=(x_1^i,x_2^i,\ldots,x_d^i)\}_{i\in \N}\subset O$, $\{n_i\}_{i\in \N}$ such that
$$\tau_d^{n_i}{\bf x^i}\rightarrow {\bf z},\quad {i\to\infty}.$$
By Claim we have that
$$\{x^i_1\}\times \ldots\times \{ x^i_{j-1}\}\times \pi^{-1}(\pi(x^i_j))\times \{x^i_{j+1}\}\times \ldots \times \{ x^i_d\}\subset L=\overline{\O}(O,\tau_d).$$
Since $\pi$ is open, it follows that
\begin{equation*}
  \begin{split}
   &\quad  \tau_d^{n_i}\Big (\{x^i_1\}\times \ldots\times \{ x^i_{j-1}\}\times \pi^{-1}(\pi(x^i_j))\times \{x^i_{j+1}\}\times \ldots \times \{ x^i_d\} \Big ) \\
    & =\{T^{n_i}x^i_1\}\times \ldots\times \{ T^{(j-1)n_i}x^i_{j-1}\}\times \pi^{-1}(\pi(T^{jn_i}x^i_j)),\{T^{(j+1)n_i}x^i_{j+1}\}\times \ldots \times \{T^{dn_i} x^i_d\} \\
    & \xrightarrow{i\to\infty} \{z_1\}\times \{z_2\}\times \ldots \times \{z_{j-1}\}\times \pi^{-1}(\pi(z_j))\times \{z_{j+1}\}\times \ldots \times \{z_d\}\\
    & \subset \overline{\O}(L,\tau_d)=L.
   \end{split}
\end{equation*}

To sum up, we have that if
${\bf z}=(z_1,z_2,\ldots, z_d)\in L=\overline{\O}(O,\tau_d),$ then for each $j\in \{1,2,\ldots, d\}$,
$$(z_1,z_2,\ldots, z_{j-1}, z_j',z_{j+1},\ldots, z_d)\in L,$$
where $\pi(z_j')=\pi(z_j)$.
Thus,  $(z_1,z_2,\ldots, z_d)\in L$ if and only if $(z_1',z_2',\ldots, z_d')\in L$ whenever for all $j\in \{1,2,\ldots, d\}$, $\pi(z_i)=\pi(z_i')$. That is,
$L= \overline{\O}(O,\tau_d)$ is $\pi^{(d)}$-saturated.

The proof is complete.
\end{proof}

\section{Proofs of Theorems B and C}

In this section we prove Theorems B and C.

\subsection{Proof of Theorem B for pro-nilsystems}\
\medskip

For a t.d.s. $(X,T)$ and subsets $U,V$ of $X$, put $N(U,V)=\{n\in \Z: U\cap T^{-n}V\not=\emptyset\}$.

\begin{lem}\label{prod} Let $(X,T)$ and $(Y,S)$ be two minimal t.d.s.
Then the maximal equicontinuous factor of $(X\times Y, T^n\times S^m)$ is $X_{eq}\times Y_{eq}$ for any $n,m\in\N$,
where $X_{eq}=X_1,Y_{eq}=Y_1$ are the maximal equicontinuous factors of $X$ and $Y$ respectively.
\end{lem}

\begin{proof}
First recall two facts used in the proof. (1) If $(Z,R)$ has a dense set of minimal point,
so does $(Z\times Z, R\times R)$. (2) If $(Z,R)$ is a t.d.s with dense minimal points and $(x,y)\in \RP(Z)$, then
for each neighborhood $U$ of $(x,y)$ and each neighborhood $W$ of the diagonal $\Delta_Z=\{(z,z):z\in Z\}$,
$N(U,W)$ is thickly syndetic.

\medskip

Now let $\pi:(X,T)\lra (X_{eq},T)$ and $\phi:(Y,S)\lra (Y_{eq},S)$ be the factor maps to the maximal equicontinuous factors
of $(X,T)$ and $(Y,S)$ respectively. Fix $n,m\in\N$. Since $\RP(X,T)=\RP(X,T^n)$ and $\RP(Y,S)=\RP(Y,S^m)$, we have that $\pi:(X,T^n)\lra (X_{eq},T^n)$ and
$\phi:(Y,S^m)\lra (Y_{eq},S^m)$ are the factor maps to the maximal equicontinuous factors
of $(X,T^n)$ and $(Y,S^m)$ respectively. Since $(X_{eq}\times Y_{eq},T^n\times S^m)$ is equicontinuous, it remains to
show that $R_{\pi\times \phi}\subset \RP(T^n\times S^m)$.

To this aim, we assume that $(x_1,y_1), (x_2,y_2)\in X\times Y$ with $(x_1,x_2)\in \RP(X,T^n)$ and $(y_1,y_2)\in \RP(Y,S^m)$. Then $\pi(x_1)=\pi(x_2)$
and $\phi(y_1)=\phi(y_2)$. Fix $\ep>0$. Let $U_1\times V_1$, $U_2\times V_2$ be neighborhoods of $(x_1,y_1)$
and $(x_2,y_2)$ respectively. Moreover, let $W_1$ and $W_2$ be the $\ep$-neighborhoods of $\Delta_X$ and $\Delta_Y$
respectively. Then both of $N_{T^n\times T^n}(U_1\times U_2,W_1)$ and $N_{S^m\times S^m}(V_1\times V_2, W_2)$ are thickly syndetic. Thus
$$ N_{T^n\times T^n}(U_1\times U_2,W_1)\cap N_{S^m\times S^m}(V_1\times V_2, W_2)\neq \emptyset.$$
Pick $k\in
N_{T^n\times T^n}(U_1\times U_2,W_1)\cap N_{S^m\times S^m}(V_1\times V_2, W_2)$. Then there are $(x_1',y_1')\in U_1\times V_1$ and $(x_2',y_2')\in U_2\times V_2$
such that
$$T^{kn}\times T^{kn}(x_1',x_2')\in W_1, \quad S^{km}\times S^{km}(y_1',y_2')\in W_2.$$
Denote the metrics of $X$ and $Y$ by $\rho_X$ and $\rho_Y$, and let the metric of $X\times Y$ be $$\rho_{X\times Y}((x_1,y_1),(x_2,y_2))=\max\{\rho_X(x_1,x_2), \rho_Y(y_1,y_2)\}.$$ Then we have that
$$\rho_{X\times Y}((T^{n}\times S^m)^k(x_1',y_1'), (T^{n}\times S^m)^k(x_2',y_2'))\le \rho_X(T^{nk}x_1',T^{nk}x_2')
+\rho_Y(S^{mk}y_1',S^{mk}y_2')<2\ep.$$
As $\ep$ is arbitrary, we have
$$\Big((x_1,x_2),(y_1,y_2)\Big)\in \RP(X\times Y, T^n\times S^m).$$
This ends the proof.
\end{proof}






\begin{lem}\label{lem-5.2}
Let $\pi:(X,\Gamma)\rightarrow  (Y,\Gamma)$ be a factor map between two minimal systems and $k\in \N$,
where $\Gamma$ is abelian. If for some $y\in Y$, $\pi^{-1}(y)^2\subset \RP^{[k]}(X)$ then $$R_\pi\subset \RP^{[k]}(X).$$
\end{lem}

\begin{proof}
Let $z\in Y$ and $x_1,x_2\in \pi^{-1}(z)$. Then by Auslander-Ellis Theorem  $(x_1,x_2)$ is proximal to some
minimal point $(y_1,y_2)$ with $\pi(y_1)=\pi(y_2)$.
Since $(Y,\Gamma)$ is minimal, it is easy to see that $$\emptyset \neq \overline{\O}((y_1,y_2),\Gamma)\cap \pi^{-1}(y)^2\subset \RP^{[k]}(X).$$
As $\overline{\O}((y_1,y_2),\Gamma)$ is minimal and $\RP^{[k]}(X)$ is $\Gamma$-invariant closed subset of $X\times X$, we have
$$\overline{\O}((y_1,y_2),\Gamma)\subset \RP^{[k]}(X).$$
In particular, $(y_1,y_2)\in \RP^{[k]}(X)$.
By ${\bf P}(X)\subset \RP^{[k]}(X)$ we have $$(x_1,y_1), (x_2,y_2)\in \RP^{[k]}(X).$$
We conclude that $(x_1,x_2)\in \RP^{[k]}(X)$ since $\RP^{[k]}(X)$ is an equivalence relation by Theorem \ref{thm-RP-d}.
\end{proof}

\begin{cor}\label{cor-5.3}
Let $\pi:(X,T)\rightarrow  (Y,T)$ be a factor map between two minimal systems and $d,k\in \N$.
Then $$\pi^{(d)}: (N_d(X), \langle\sigma_d, \tau_d\rangle)\rightarrow (N_d(Y),\langle\sigma_d, \tau_d\rangle)$$
is a factor map. If for some $x\in X$
$$\{x^{(d)}\}\times \Big(\pi^{-1}(\pi(x))\Big)^d\subset \RP^{[k]}(N_d(X),\langle\sigma_d, \tau_d\rangle),$$
Then $$R_{\pi^{(d)}}\subset \RP^{[k]}(N_d(X),\langle\sigma_d, \tau_d\rangle).$$
\end{cor}

\begin{proof}
Let $y=\pi(x)$ and $\G_d=\langle\sigma_d, \tau_d\rangle$. Note that $(N_d(X,T), \G_d)$ is minimal. By Lemma \ref{lem-5.2}, we need to show that
$$ \Big((\pi^{(d)})^{-1}(y^{(d)})\Big)^2\subset \RP^{[k]}(N_d(X),\langle\sigma_d, \tau_d\rangle). $$
Let ${\bf x}, {\bf x'}\in (\pi^{(d)})^{-1}(y^{(d)}) $. Since $x^{(d)}\in N_d(X)$, by our assumption
$$(x^{(d)},{\bf x}), (x^{(d)},{\bf x'})\in \{x^{(d)}\}\times \Big(\pi^{-1}(\pi(x))\Big)^d\subset \RP^{[k]}(N_d(X),\langle\sigma_d, \tau_d\rangle).$$
Since $\RP^{[k]}$ is an equivalence relation (Theorem \ref{thm-RP-d}),
$$({\bf x},{\bf x'})\in \RP^{[k]}(N_d(X),\langle\sigma_d, \tau_d\rangle).$$
That is,
$$ \Big((\pi^{(d)})^{-1}(y^{(d)})\Big)^2\subset \RP^{[k]}(N_d(X),\langle\sigma_d, \tau_d\rangle). $$
The proof is complete.
\end{proof}

We will prove the following theorem.
\begin{thm} \label{proof-nil}Let $(X,T)$ be a minimal pro-nilsystem. Then the maximal equicontinuous
factor of $(N_d(X), \langle\sigma_d, \tau_d\rangle)$
is $(N_d(X_{eq}), \langle\sigma_d, \tau_d\rangle)$.
\end{thm}

\begin{proof}
Let $\pi_1:X\lra X_{1}=X_{eq}$ be the factor map to the maximal equicontinuous factor and $d\in \N$.
Let $\G_d=\langle\sigma_d, \tau_d\rangle$. Then $\pi_1$ induces a factor map
$$\pi^{(d)}_{1}: (N_d(X),\G_d)\lra (N_d(X_{1}),\G_d).$$
Since $(X_1,T)$ is equicontinuous, so is $(N_d(X_1),\G_d)$. By Theorem \ref{thm-RP} it
follows that $$\RP((N_d(X),\G_d))\subset R_{\pi_1^{(d)}}.$$
To show that the maximal equicontinuous
factor of $(N_d(X),\G_d)$ is $(N_d(X_{1}),\G_d)$, it remains to show that
\begin{equation}\label{y5.1}
  R_{\pi_1^{(d)}}\subset \RP(N_d(X),\G_d).
\end{equation}
For simplicity, we use
$$(x_1,\ldots,x_d)\underset{\G_d}\sim (y_1,\ldots,y_d), \ \text{and}\ (x_1,\ldots,x_d)\underset{\tau_d}\sim (y_1,\ldots,y_d)$$
to denote $\big((x_1,\ldots,x_d),(y_1,\ldots,y_d)\big)\in \RP(N_d(X),\G_d)$ and
$\big((x_1,\ldots,x_d),(y_1,\ldots,y_d)\big)\in \RP(N_d(X),\tau_d)$ respectively.

\medskip

We prove \eqref{y5.1} by induction on $d$.

\medskip
\noindent{\bf Step 1}:
For $d=1$ it is clear. For $d=2$, $\G_2$ is generated by $T\times T$ and $\id\times T$.
It is clear that $N_2(X)=X\times X$
and $N_2(X_{1})=X_{1}\times X_{1}$. By Lemma~\ref{prod} if $\pi_1(x_1)=\pi_1(x_2)$ and $\pi_1(y_1)=\pi_1(y_2)$
then $((x_1,y_1),(x_2,y_2))$ is regionally proximal for $T\times T^2$, and thus $((x_1,y_1),(x_2,y_2))\in \RP(N_2(X),\G_2)$.

\medskip
\noindent {\bf Step 2:}
Now we assume that \eqref{y5.1} holds for $d-1$, where $d\ge 3$. Let $$p_1: (N_d(X),\G_d)
\rightarrow (N_{d-1}(X),\G_{d-1}), (x_1,\ldots,x_{d-1},x_d)\mapsto (x_1,\ldots,x_{d-1})$$
be the projection to the first $d-1$ coordinates.
By Lemma \ref{lem-5.2}, we need to show that for some $z\in X_1$,
$$(\pi_1^{(d)})^{-1}(z^{(d)})\times (\pi_1^{(d)})^{-1}(z^{(d)})\subset \RP(N_d(X),\G_d).$$
Let $$(x_1,x_2,\ldots,x_d), (y_1,y_2,\ldots,y_d)\in (\pi_1^{(d)})^{-1}(z^{(d)}).$$
Then by mapping $p_1$, $(x_1,\ldots,x_{d-1}), (y_1,\ldots,y_{d-1})\in N_{d-1}(X)$ and
$(x_1,\ldots,x_{d-1}), (y_1,\ldots,y_{d-1})\in (\pi_1^{(d-1)})^{-1}(z^{(d-1)})$.
By the inductive hypothesis, $\Big((x_1,\ldots,x_{d-1}), (y_1,\ldots,y_{d-1})\Big)\in \RP(N_{d-1}(X),\G_{d-1})$, i.e.
$$(x_1,\ldots,x_{d-1})\underset{\G_{d-1}}\sim (y_1,\ldots,y_{d-1}).$$
By Theorem \ref{thm-RP}, there is some $x_d',y_d'\in X$ such that
$(x_1,\ldots,x_{d-1},x_d'), (y_1 ,\ldots, y_{d-1}, y_d')\in N_d(X)$ with
$$(x_1,\ldots,x_{d-1},x_d')\underset{\G_{d}}\sim(y_1 ,\ldots, y_{d-1}, y_d').$$
If we can show that $$(x_1,\ldots,x_{d-1},x_d)\underset{\G_{d}}\sim(x_1,\ldots,x_{d-1},x_d'),
\quad (y_1 ,\ldots, y_{d-1}, y_d)\underset{\G_{d}}\sim (y_1 ,\ldots, y_{d-1}, y_d'),$$
then by equivalence of $\RP$, we have that
$$(x_1,\ldots,x_{d-1},x_d)\underset{\G_{d}}\sim(y_1 ,\ldots, y_{d-1}, y_d),$$
that is, \eqref{y5.1} holds for $d$. So we only need to show that
$(x_1,\ldots,x_{d-1},x_d)\underset{\G_{d}}\sim(x_1,\ldots,x_{d-1},x_d')$,
and similarly we will have $(y_1 ,\ldots, y_{d-1}, y_d)\underset{\G_{d}}\sim (y_1 ,\ldots, y_{d-1}, y_d')$.

Since $(x_1,\ldots,x_{d-1},x_d)\in N_d(X)$, for a fixed $x\in \pi_1^{-1}(z)$,
there is some sequence $\{g_i\}_{i\in \N}\subset \G_d$ such that
$$g_i(x_1,\ldots,x_{d-1},x_d)\to x^{(d)}, \quad i\to\infty.$$
Without loss of generality, we assume that
$$g_i(x_1,\ldots,x_{d-1},x_d')\to (x^{(d-1)},y), \quad i\to\infty,$$ for some $y\in X$.
If $x^{(d)} \underset{\G_{d}}\sim (x^{(d-1)},y)$, then by $$\Big ( (x_1,\ldots,x_{d-1},x_d),(x_1,\ldots,x_{d-1},x_d')\Big)\in
\overline{\O}\Big(\big(x^{(d)}, (x^{(d-1)},y)\big), \G_d^{(2)} \Big)$$ we will have
$(x_1,\ldots,x_{d-1},x_d)\underset{\G_{d}}\sim(x_1,\ldots,x_{d-1},x_d')$, where $\G^{(2)}_d=\{(g,g): g\in \G_d\}$.
Thus it left to show that
$$x^{(d)} \underset{\G_{d}}\sim (x^{(d-1)},y).$$

By Lemma \ref{sharpnumber} and  $(x^{(d-1)},y)\in N_d(X)$, $(x,y)\in \RP^{[d-2]}(X,T)$. Thus
$(x,y)\in \RP_{\pi_{d-3}}(X,T)$ by Lemma \ref{qiu}, where $\pi_{d-3}: X\rightarrow X_{d-3}$.
Note that when $d=3$, $X_0$ is the trivial system and $\RP_{\pi_0}(X,T)=\RP(X,T)$.

We proceed to prove that
$$x^{(d-1)}\underset{\tau_{d-1}}\sim (x^{(d-2)},y).$$
If $d=3$, then by Step 1, $(x,x)\underset{T\times T^2}\sim (x,y)$. If $d\ge 4$, then we prove as follows.
Let $\ep>0$.  Since $(x,y)\in \RP_{\pi_{d-3}}(X,T)=\RP_{\pi_{d-3}}(X,T^{d-1})$,
there are $x',y'\in X$, $w\in X_{d-3}$ and $n\in \Z$ such that
$$\rho(x,x')<\ep, \rho(y,y')<\ep, \pi_{d-3}(x')=\pi_{d-3}(y')=w,\ \text{and}\  \rho(T^{(d-1)n}x',T^{(d-1)n}y')<\ep.$$
By Lemma \ref{eli-thm}
we have that $(\pi_{d-3}^{(d-1)})^{-1}N_{d-1}(X_{d-3})=N_{d-1}(X)$. Since $w^{(d-1)}\in N_{d-1}(X_{d-3})$, we deduce
$$\Big(\pi_{d-3}^{-1}(w)\Big)^{d-1}= (\pi_{d-3}^{(d-1)})^{-1} (w^{(d-1)})\subset N_{d-1}(X).$$
This implies that $(x')^{(d-1)},((x')^{(d-2)},y')\in N_{d-1}(X)$.
To sum up, for each $\ep>0$, there are $(x')^{(d-1)},((x')^{(d-2)},y')\in N_{d-1}(X)$ and $n\in \Z$
such that $\rho_{d-1}\big(x^{(d-1)} , (x')^{(d-1)}\big)<\ep$, $\rho_{d-1}\big((x^{(d-2)},y) , ((x')^{(d-2)},y') \big)<\ep $  and
$$\rho_{d-1}\Big(\tau_{d-1}^n\big((x')^{(d-1)}\big), \tau_{d-1}^n((x')^{(d-2)},y')\Big) <\ep,$$
which implies that $$x^{(d-1)}\underset{\tau_{d-1}}\sim (x^{(d-2)},y).$$

We continue our proof.
For each $\ep>0$, let $U, V$ be open
neighborhoods of $x^{(d)}$ and $(x^{(d-1)},y)$ in $N_d(X)$
with ${\rm diam}(U)<\ep/2$, ${\rm diam}(V)<\ep/2$ respectively.
Let $$p_{2}:(N_d(X),\G_d)\rightarrow (N_{d-1}(X),\G_{d-1}), (x_1,x_2, \ldots, x_d)\mapsto (x_2,\ldots,x_{d})$$
be the projection to the last $d-1$ coordinates.
Then $p_2 (U),p_2(V)$ be open neighborhoods of $x^{(d-1)}$ and $(x^{(d-2)},y)$ in $N_{d-1}(X)$ respectively.
Since $x^{(d-1)}\underset{\tau_{d-1}}\sim (x^{(d-2)},y)$, there are ${\bf y} \in p_2(U)$, ${\bf y'} \in p_{2}(V)$
and $n\in \Z$ such that
$$\rho_{d-1}\big(\tau_{d-1}^n({\bf y}), \tau_{d-1} ^n({\bf y'})\big)<\ep.$$ There are $y_1,y_1'\in X$ such that $(y_1,{\bf y})\in U$ and
$(y_1',{\bf y'}) \in V$. As ${\rm diam}(U)<\ep/2$, ${\rm diam}(V)<\ep/2$, it follows that $\rho (y_1,y_1')<\ep$. This implies that
\begin{equation}\label{equi-2}
\rho_d \Big((\id\times \tau_{d-1}^n(y_1,{\bf y}), (\id\times \tau_{d-1}^n(y_1',{\bf y'})\Big)<\ep.
\end{equation}
This shows that $$x^{(d)} \underset{\G_{d}}\sim (x^{(d-1)},y),$$
since $\G_d$ is also generated by $\sigma_d$ and $\id \times \tau_{d-1}$.
The whole proof is complete.
\end{proof}

\subsection{Proof of Theorem B for general systems}\
\medskip

Actually, we will show more. First we need a lemma.

\begin{lem}\label{lem-5.5}
Let $\pi: (X,\Gamma)\rightarrow (Y,\Gamma)$ be a factor map between minimal systems and $k\in \N$,
where $\Gamma$ is abelian. If $\pi$ is proximal, then the maximal $k$-step pro-nilfactor of $(X,\Gamma)$
is the same as the one of $(Y,\Gamma)$, i.e. $X_k(X)=X_k(Y)$.
\end{lem}

\begin{proof}
First we have the following commutative diagram:
\[
\begin{CD}
X @>{\pi_{k,X}}>> X_k(X)\\
@V{\pi}VV      @VV{\pi'}V\\
Y @>>{\pi_{k,Y}}> X_k(Y)
\end{CD}
\]

We need to show that $X_k(X)=X_k(Y)$. Otherwise there are $x_1'\neq x_2'\in X_k(X)$ such that
$\pi'(x'_1)=\pi'(x'_2)=z$. Choose $x_1,x_2\in X$ such that $\pi_{k,X}(x_1)=x_1', \pi_{k,X}(x_2)=x_2'$.
Let $y_1=\pi(x_1), y_2=\pi(x_2)$. Since $\pi'\circ \pi_{k,X}=\pi_{k,Y}\circ \pi$,
we have $\pi_{k,Y}(y_1)=\pi_{k,Y}(y_2)=z$, i.e. $(y_1,y_2)\in \RP^{[k]}(Y)$.

By Theorem \ref{lift}, there are $(\widetilde{x_1},\widetilde{x_2})\in \RP^{[k]}(X)$
such that $\pi\times \pi(\widetilde{x_1},\widetilde{x_2})=(y_1,y_2)$. Since $\pi$ is proximal,
$(\widetilde{x_1},x_1),(\widetilde{x_2},x_2)\in {\bf P}(X)$. Since ${\bf P}\subset \RP^{[k]}$, by Theorem \ref{thm-RP-d}
$$(x_1,x_2)\in \RP^{[k]}(X).$$
It follows that $x_1'=\pi_{k,X}(x_1)=\pi_{k,X}(x_2)=x_2'$, a contradiction.
\end{proof}

\begin{thm}\label{c-1-1}
Let $(X,T)$ be minimal and $d\in\N$. Then for each $k\in\N$
the maximal $k$-step pro-nilfactor of $(N_d(X), \langle\sigma_d, \tau_d\rangle)$
is the same as the one of $(N_d(X_\infty), \langle\sigma_d, \tau_d\rangle)$.
Moreover, the maximal equicontinuous factor of $(N_d(X), \langle\sigma_d, \tau_d\rangle)$
is $(N_d(X_1), \langle\sigma_d, \tau_d\rangle)$, where $X_1=X_{eq}$ is the maximal equicontinuous factor of $X$.
\end{thm}

\begin{proof} If $X$ is equicontinuous, then the theorem holds. Thus, we assume that $X$ is not equicontinuous.
Let $X_\infty$ be the $\infty$-step pro-nilfactor of $X$. Then we have the following diagram
\[
\begin{CD}
X @<{\sigma^*}<< X^*\\
@V{\pi_\infty}VV      @VV{\pi^*}V\\
X_\infty @<{\tau^*}<< X_\infty^*
\end{CD}
\]
where $\sigma^*,\tau^*$ are almost one to one and $\pi^*$ is open by Theorem A. By Theorem A
we know that there is a dense $G_\delta$ subset $\Omega$ of $X^*$ such that for each $x\in \Omega$,
and each $l\in \N$, the orbit closure of $x^{(l)}$ under $\tau_l$ is $(\pi^*)^{(l)}$-saturated.

The diagram above induces the following commutative diagram:
\[
\begin{CD}
(N_d(X),\G_d) @<{(\sigma^*)^{(d)}}<< (N_d(X^*),\G_d)\\
@V{\pi_\infty^{(d)}}VV      @VV{(\pi^*)^{(d)}}V\\
(N_d(X_\infty),\G_d) @<{(\tau^*)^{(d)}}<<(N_d( X_\infty^*),\G_d)
\end{CD}
\]
where $(\sigma^*)^{(d)}, (\tau^*)^{(d)} $ are almost one to one, $\G_d=\langle\sigma_d, \tau_d\rangle$.
By Lemma \ref{lem-5.5}, for $k\in \N$, to show that the maximal $k$-step pro-nilfactor of $(N_d(X), \G_d)$
is the same as the one of $(N_d(X_\infty), \G_d)$, it suffices to show that the maximal $k$-step pro-nilfactor of $(N_d(X^*), \G_d)$
is the same as the one of $(N_d(X^*_\infty), \G_d)$.
To that aim, we need to show that
$$R_{(\pi^*)^{(d)}}\subset \RP^{[k]}(N_d(X^*), \G_d).$$
By Corollary \ref{cor-5.3}, we show that for some $x\in X^*$
$$\{x^{(d)}\}\times \Big((\pi^*)^{-1}(\pi^*(x))\Big)^d\subset \RP^{[k]}(N_d(X^*),\G_d),$$

Let $x\in \Omega$ and let $(z_1,\ldots, z_d)\in \Big((\pi^*)^{-1}(\pi^*(x))\Big)^d $.
We show $(x^{(d)},(z_1,\ldots, z_d))\in \RP^{[k]}(N_d(X^*), \G_d)$.

For each $\ep>0$, let $U_i=B_\ep(x)$ and $V_i=B_\ep(z_i)$, $1\le i\le d$.
Then there is $n\in\mathbb{Z}$ such that
$${\bf x}=(T^{n}x, T^{2n}x,\ldots, T^{dn}x)\in U_1\times U_2\times \ldots\times U_d,$$
$$\tau_d^{dnj}({\bf x})\in V_1\times V_2\times\ldots\times V_d, 1\le j\le k+1.$$

Now we explain why we can do this. If $x^{(d)}=(z_1,\ldots, z_d)$ we just put $n=0$, otherwise we may assume that $n\not=0$.
For $1\le j \le k+1$, and $1\le i_1<i_2\le d$, the $i_1$-th and $i_2$-th coordinates of $\tau_d^{dnj}({\bf x})$
are $T^{dnji_1+i_1n}x$ and $T^{dnji_2+i_2n}x$ respectively. It is clear they are distinct.

For $1\le j_1<j_2\le k+1$ and $1\le i_1\le i_2\le d$, the $i_1$-th coordinate of $\tau_d^{dnj_1}({\bf x})$ and $i_2$-th coordinate of $\tau_d^{dnj_2}({\bf x})$
are $T^{dnj_1i_1+i_1n}x$ and $T^{dnj_2i_2+i_2n}x$ respectively. It is clear they are also distinct, since $j_1i_1=j_2i_2$ implies $i_1\not=i_2$.

Let ${\bf y}=\tau_d^{dn}({\bf x})$. We have $$\rho_d(\tau_d^{dnj}({\bf x)},\tau_d^{dnj}({\bf y}))\le d\ep\ \text{ for} \ j=1,2,\ldots,k.$$
This implies that $(x^{(d)}, (z_1,\ldots, z_d))\in \RP^{[k]}(N_d(X^*),\G_d)$ by the definition.
Thus we have proved the maximal $k$-step nilfactor factor of $(N_d(X^*), \G_d)$
is the same as the maximal $k$-step nilfactor of $(N_d(X^*_\infty), \G_d)$.
It follows that the maximal $k$-step nilfactor factor of $(N_d(X), \G_d)$
is the same as the maximal $k$-step nilfactor of $(N_d(X_\infty),\G_d)$.

\medskip

Moreover, the maximal equicontinuous factor of $(N_d(X), \G_d)$
is  same as the maximal equicontinuous factor of $(N_d(X_\infty),\G_d)$.
By Theorem \ref{proof-nil},  the maximal equicontinuous
factor of $(N_d(X_\infty), \G_d)$
is $(N_d(X_{1}), \G_d)$. Thus the maximal equicontinuous factor of $(N_d(X), \G_d)$
is $(N_d(X_1), \G_d)$. The proof is complete.
\end{proof}

A similar proof yields the following result.

\begin{thm} \label{c-1-1-distal}
Let $(X,T)$ be a minimal t.d.s. with $d, k \in \N$. Then there is a dense $G_\delta$ set $\Omega$ such that for each $x\in \Omega$,
the maximal $k$-step pro-nilfactor of $\overline{\O}(x^{(d)},\tau_d)$ is the same as the one of
$\overline{\O}((\pi_\infty x)^{(d)},\tau_d)$.
\end{thm}

\begin{proof} The proof is a modification of the previous one.  If $X$ is equicontinuous, then the theorem holds. Thus, we assume that $X$ is not equicontinuous.

Assume first we have the same diagram as
in the proof of Theorem \ref{c-1-1}. By Theorem \ref{key} we know that there is a dense $G_\delta$
subset $\Omega^*$ of $X^*$ such that for each $x\in \Omega^*$, and each $l\in \N$, the orbit closure
of $x^{(l)}$ under $\tau_l$ is $(\pi^*)^{(l)}$-saturated. By the same analysis as
in
the proof of
Theorem \ref{c-1-1}, we show that for each $x\in \Omega^*$,
the maximal $k$-step pro-nilfactor  of $\overline{\O}(x^{(d)},\tau_d)$ is the same as the one of
$\overline{\O}((\pi^* x)^{(d)},\tau_d)$.

\medskip

The result is clear for $d=1$. We now assume that $d\ge 2$. Assume that $(x_1,\ldots,x_d)$,
$(y_1,\ldots,y_d)$ $\in$ $\overline{\O} (x^{(d)},\tau_d)$ with $\pi^*(x_i)=\pi^*(y_i)$ for a given $x\in \Omega^*$.
We will show that $$\Big((x_1,\ldots,x_d),(y_1,\ldots, y_d)\Big)\in \RP^{[k]}(\overline{\O} (x^{(d)},\tau_d), \tau_d)$$ for each $k\in\N$.

To do this, for a given $k\in\N$ and each $\ep>0$, let $U_i=B_\ep(x_i)$ and $V_i=B_\ep(y_i)$, $1\le i\le d$.
Then there is $n\in\Z$ such that
$${\bf x}=(T^{n}x, \ldots, T^{dn}x)\in U_1\times U_2\times \ldots\times U_d,$$
$$\tau_d^{dnj}({\bf x})\in V_1\times V_2\times\ldots\times V_d, 1\le j\le k+1,$$
Let ${\bf y}=\tau_d^{dn}({\bf x})$. Then ${\bf y}=\tau_d^{dn+n}(x^{(d)})\in \overline{\O} (x^{(d)},\tau_d)$ and  $$\rho_d(\tau_d^{dnj}({\bf x)},\tau_d^{dnj}({\bf y}))\le d\ep\ \text{ for} \ j=1,2,\ldots,k.$$
%
This implies that $\Big((x_1,\ldots,x_d),(y_1,\ldots, y_d)\Big)\in \RP^{[k]}(\overline{\O} (x^{(d)},\tau_d), \tau_d)$
by definition of $\RP^{[k]}$. The proof is complete.
\end{proof}
We remark that in general $\overline{\O} (x^{(d)},\tau_d)$ is not minimal. Thus, to show Theorem \ref{c-1-1-distal}
we can not use exactly the same arguments as we used in Theorem \ref{c-1-1}.

\subsection{Proof of Theorem C}\
\medskip

First we have the following simple observation.

\begin{lem} Let $(X,T)$ be a minimal system. We have
\begin{enumerate}
\item $N_d(X, T)=N_d(X, T^{-1})$ and $N_d(X, T^n)\subset N_d(X, T)$ for any $n\in\Z\setminus\{0\}$.


\item $\displaystyle N_d(X, T)=\bigcup_{l,k=0}^{n-1} \sigma_d^l\tau_d^k N_d(X, T^n)$.
\end{enumerate}
\end{lem}

The following proposition will be used in the proof of Theorem C.

\begin{prop} \label{t-n} If $(X,T)$ is a minimal equicontinuous system and
$(X,T^n)$ is minimal for some $n\in\N$, then for any $d\in \N$,
$N_d(X, T)=N_d(X, T^n)$.
\end{prop}

\begin{proof}
Let $d\in \N$. It is clear that $N_d(X, T^n)\subset N_d(X, T)$. Now we show that $N_d(X, T)\subset N_d(X, T^n)$.  Let $x\in X$. It suffices to show that
$$\O(x^{(d)}, \langle\sigma_d, \tau_d\rangle)= \O((x,\ldots,x),\langle\sigma_d, \tau_d\rangle) \subset N_d(X, T^n).$$
Since each point in $\O(x^{(d)},\langle\sigma_d, \tau_d\rangle)$
has the form of $$(T^{k+l}x,T^{k+2l}x,\ldots, T^{k+dl}x)$$ for some $k,l\in\Z$, we need to show that
$$(T^{k+l}x,T^{k+2l}x,\ldots, T^{k+dl}x)\in N_d(X, T^n).$$

By the assumption that $(X,T^n)$ is minimal, there are sequences $\{p_i\}_{i\in \N}, \{q_i\}_{i\in \N}\subset \Z$ such that
$$T^{np_i}x\lra T^{k}x , \quad T^{nq_i}\rightarrow T^lx, \quad i\to\infty.$$ Let $\ep>0$. Since $(X,T)$ is equicontinuous, there is $\delta>0$
such that if $\rho(x,y)<\delta$ then $\rho(T^ix,T^iy)<\ep$ for all $i\in\Z$.

Since $T^{nq_i}x\rightarrow T^lx, i\to\infty$,
there is some $N\in\N$ such that if $i\ge N$ then $\rho(T^{nq_i}x,T^lx)<\delta$.
This implies that for any $1\le j\le d$ and $i\ge N$ we have
$$\rho(T^{jnq_i}x,T^{(j-1)nq_i+l}x)<\ep,  \rho(T^{(j-1)nq_i+l}x,T^{(j-2)nq_i+2l}x)<\ep,  \ldots,    \rho(T^{nq_i+(j-1)l}x,T^{jl}x)<\ep$$
which implies that $$\rho(T^{jnq_i}x,T^{jl}x)<j\ep\le d\ep, \quad \forall 1\le j\le d.$$ Since $\ep$ is arbitrary,
$$(T^{nq_i}x,T^{n2q_i}x,\ldots,T^{ndq_i}x)\to (T^lx,T^{2l}x,\ldots,T^{dl}x), \quad i\to\infty.$$
It follows that there are $p_i',q_i'\in \Z$ with
$$(T^{np_i'+nq_i'}x,T^{np_i'+n2q_i'}x,\ldots,T^{np_i'+ ndq_i'}x)\to (T^{k+l}x,T^{k+2l}x,\ldots,T^{k+dl}x), \quad i\to\infty.$$
Thus
$$(T^{k+l}x,T^{k+2l}x,\ldots, T^{k+dl}x)\in N_d(X, T^n).$$ The proof is complete.
\end{proof}

\begin{thm}\label{thm-5.10}
Let $G$ be an abelian group and $\Gamma$ be its subgroup with finite index, i.e. $[G: \Gamma]<\infty$. Let $(X,G)$ be a minimal system, and let $\pi: (X, G)\rightarrow (X_{eq}, G)$ be the factor map to its maximal equicontinuous factor. Then $(X,\Gamma)$ is minimal if and only if $(X_{eq},\Gamma)$ is minimal.
\end{thm}

\begin{proof}
Since $\pi: (X,\Gamma)\rightarrow (X_{eq},\Gamma)$ is a factor map, the minimality of $(X,\Gamma)$ implies the minimality of $(X_{eq},\Gamma)$. Now we show the converse.

Assume that  $(X_{eq},\Gamma)$ is minimal, and we will show that $(X,\Gamma)$ is minimal. Note that $\pi: (X, G)\rightarrow (X_{eq}, G)$ is the factor map to its maximal equicontinuous factor. Since $[G:\Gamma]<\infty$, it follows that $\pi: (X,\Gamma)\rightarrow (X_{eq}, \Gamma)$ is also the factor map to its maximal equicontinuous factor. Thus  any equicontinuous factor of $(X,\Gamma)$ is also a factor of $(X_{eq},\Gamma)$. In particular, any equicontinuous factor of $(X,\Gamma)$ is minimal as $(X_{eq},\Gamma)$ is minimal.

If $(X,\Gamma)$ is not minimal, then there is a non-empty $\Gamma$-minimal subset $W$ of $(X,\Gamma)$ with $W\neq X$. Since $[G: \Gamma]<\infty$, there are $h_1,h_2,\ldots, h_m\in G, m\in \N$ such that $$G=\bigcup_{i=1}^m h_i \Gamma.$$
Since $(W,\Gamma)$ is minimal and $G$ is abelian, $(h_iW,\Gamma)$ is also minimal for all $i\in \{1,2,\ldots, m\}$. Note that $\bigcup_{i=1}^m h_iW$ is $G$-invariant, and we have $X=\bigcup_{i=1}^m h_iW$ as $(X,G)$ is minimal. Since minimal subsets are either identical or disjoint, there is subset $\{g_1,\ldots, g_r\}\subset \{h_1,\ldots, h_m\}$, $2\le r\le m$ such that
$$X=\bigsqcup_{i=1}^rg_i W,$$
where $\bigsqcup$ means disjoint union. Now define
$$\phi: (X,\Gamma)\rightarrow (\{1,2,\ldots,r\},\Gamma), \ g_iW\mapsto \{i\}, \forall i\in \{1,2,\ldots, r\}.$$
Since $(g_iW,\Gamma)$ is minimal for $i\in \{1,2,\ldots, r\}$, $\Gamma=\id$ on $\{1,2,\ldots, r\}$. As $r\ge 2$, we have that $(\{1,2,\ldots,r\},\Gamma)$ is a non-minimal equicontinuous factor of $(X,\Gamma)$, which is a contradiction. Thus $(X,\Gamma)$ is minimal.
The proof is complete.
\end{proof}

Now we are ready to show Theorem C.

\begin{proof}[Proof of Theorem C]
Let $(X,T)$ be minimal and $k\ge 2$. It suffices to show that if $(X,T^k)$ is minimal, then $N_d(X, T)=N_d(X, T^k)$ for each $d\in\N$.
This is an application of Theorem B  and some previous results.

Set $\mathcal{G}_d(T)=\langle\sigma_d(T), \tau_d(T)\rangle$.
Then $\G_d(T^k)$ is a subgroup of $\G_d(T)$ with finite index. Since $(X,T^k)$ is minimal, we have that $(X_{eq},T^k)$ is minimal. Thus, by Proposition \ref{t-n},
it follows that $N_d(X_{eq},T)=N_d(X_{eq},T^k)$ and $(N_d(X_{eq}, T), \G_d(T^k))$ is minimal. By Theorem \ref{thm-5.10},
$(N_d(X,T),\G_d(T^k))$ is minimal. Since $(X,T^k)$ is minimal, $(N_d(X,T^k), \mathcal{G}_d(T^k))$ is minimal.
It follows that $N_d(X, T)=N_d(X, T^k)$.
\end{proof}

\section{Proofs of Theorems D,E,F}

In this section we prove Theorems D,E and F. First we need a lemma.

\subsection{A key lemma}\
\medskip

To show Theorem D we need the following Lemma \ref{elibenjy}. Given a compact metric space $Z$ and a sequence of non-empty closed subsets $A_n\subset Z$, the
sets $ \liminf_{n\to\infty} A_n$ and $ \limsup_{n\to\infty} A_n$ are defined by:
$$\liminf_{n\to\infty} A_n=\{z\in Z: \exists z_n\in A_n, \ s.t.\  z=\lim_{n\to\infty} z_n\},$$
and
$$\limsup_{n\to\infty} A_n=\{z\in Z: \text{for some subsequence } \ \{n_i\}, \exists z_i\in A_{n_i}, \ s.t.\  z=\lim_{n\to\infty} z_i\}.$$
When $\displaystyle A:=\liminf_{n\to\infty} A_n=\limsup_{n\to\infty} A_n$, we write $\displaystyle  A=\lim_{n\to\infty} A_n$ and call $A$ the limit of the sequence $\{A_n\}_{n\in \N }$. In fact, in this case the set $A$ is the limit of the
sequence $\{A_n\}_{n\in \N }$ in the space $2^Z$, comprising the non-empty closed subsets of $Z$, with respect to
the Hausdorff metric.

\begin{lem}\label{elibenjy}
Let  $(X,T)$ be a t.d.s. Assume that for some $d\in \N$ and $n\in \N$,
$$N_{d+1}(X, T)=N_{d+1}(X, T^n),$$
then the subset
$$\Omega_d=\{x\in X: \overline{\O}(x^{(d)}, \tau_d(T))=\overline{\O}(x^{(d)}, \tau_d(T^n))\}$$
is a dense $G_\d$ subset of $X$, where $x^{(d)}=(x,x,\ldots,x)\in X^d, \tau_d(T)=T\times T^2\times \ldots\times T^d$.
\end{lem}

\begin{proof}
Let $\tau'_{d+1}(T)={\rm id}\times T\times T^2\times \ldots \times T^d={\rm id}\times \tau_d(T)$.
Note that
$$N_{d+1}(T)=\overline{\O}(\D_{d+1}, \tau_{d+1}(T))=\overline{\O}(\D_{d+1}, \tau'_{d+1}(T)).$$
For $x\in X$, let
$$C(x)=\overline{\O}(x^{(d+1)}, \tau'_{d+1}(T))=\{x\}\times \overline{\O}(x^{(d)}, \tau_{d}(T))$$
$$D(x)=\{(x,u_1,\ldots, u_d): \exists x_i\in X, n_i\in \Z, (\tau'_{d+1}(T))^{n_i}(x_i^{(d+1)})\to (x,u_1,\ldots, u_d) \}.$$
Then it is clear that $C(x)\subset D(x)=N_{d+1}(T)\cap \left(\{x\}\times X^d\right)$.

\medskip

\noindent {\rm Claim 1}: {\em The map $C: X\rightarrow 2^{N_{d+1}(T)}, x\mapsto C(x)$ is lower-semi-continuous, that is, $x_i\to x , i\to\infty$ implies that $\displaystyle \liminf_{i\to\infty} C(x_i)\supset C(x)$.}

\medskip

In fact, by definition it follows from $x_i\to x$ and $x_i^{(d+1)}\in C(x_i)$ for all $i$ that
$$x^{(d+1)}\in \liminf_{i\to\infty} C(x_i).$$ Now for each $k\in \Z$, since $(\tau'_{d+1}(T))^k(x_i^{(d+1)})\in C(x_i)$ for all $i$, one has that
$$(\tau'_{d+1}(T))^k(x^{(d+1)})\in \liminf_{i\to\infty}C(x_i).$$
Thus
$${\O}(x^{(d+1)}, \tau'_{d+1}(T))\subset \liminf_{i\to\infty}C(x_i) .$$ It follows that
$$ \overline{\O}(x^{(d+1)}, \tau'_{d+1}(T))\subset \liminf_{i\to\infty}C(x_i) .$$
as claimed.

\medskip

The following claim is a direct consequence of the definition of $D$.

\medskip

\noindent {\rm Claim 2}: {\em For every $x\in X$,
$$D(x)\subset \bigcup\{\liminf_{i\to\infty}C(x_i): x_i\to x\}.$$}

Let $X_0\subset X$ be the set of $C$-continuity points. Then it is well known that $X_0$ is a
dense $G_d$ subset of $X$ \cite{C}. Now for $x_0\in X_0$, we have that for every
sequence $x_i \to x_0, i\to\infty$, $\displaystyle \liminf_{i\to\infty}C(x_i)=\lim_{i\to\infty}C(x_i)=C(x_0)$.
It follows that $D(x_0)\subset C(x_0)\subset D(x_0)$, whence
\begin{equation}\label{gw1}
  D(x_0)=C(x_0).
\end{equation}
In the following, we will consider both $\tau'_{d+1}(T)$ and $\tau'_{d+1}(T^n)$, and the symbols
$C_{\tau'_{d+1}(T)},$ $D_{\tau'_{d+1}(T)},$ $D_{\tau'_{d+1}(T^n)},$ and $C_{\tau'_{d+1}(T^n)}$ are clear.

\medskip

By assumption
\begin{equation*}
\begin{split}
  N_{d+1}(T) &= \bigcup_{x\in X}N_{d+1}(T)\cap \left( \{x\}\times X^d\right)\\
    & = \bigcup_{x\in X} D_{\tau'_{d+1}(T)}(x)=\bigcup_{x\in X} D_{\tau'_{d+1}(T^n)}(x)=N_{d+1}(T^n)
\end{split}
\end{equation*}
and therefore $D_{\tau'_{d+1}(T)}(x)=D_{\tau'_{d+1}(T^n)}(x)$ for all $x\in X$.
Now let $\Omega_d$ be the intersection of the sets of continuity points of the maps $C_{\tau'_{d+1}(T)}$ and  $C_{\tau'_{d+1}(T^n)}$. Then by (\ref{gw1}) for each $x\in \Omega_d$,
$$ C_{\tau'_{d+1}(T)}(x)=D_{\tau'_{d+1}(T)}(x)=D_{\tau'_{d+1}(T^n)}(x)=C_{\tau'_{d+1}(T^n)}(x). $$
Let $\pi: X^{d+1}\rightarrow X^d$ be projection on the coordinates $\{2,3,\ldots,d+1\}$, and we have that
$$\overline{\O}(x^{(d)}, \tau_d(T))=\pi C_{\tau'_{d+1}(T)}.$$
Hence for each $x\in \Omega_d$,
$$\overline{\O}(x^{(d)}, \tau_d(T))=\pi C_{\tau'_{d+1}(T)}=\pi C_{\tau'_{d+1}(T^n)}=\overline{\O}(x^{(d)}, \tau_d(T^n)).$$
The proof is completed.
\end{proof}

By Lemma \ref{elibenjy}  we have

\begin{thm}\label{equivalence-all} Let $(X,T)$ be a minimal t.d.s. and $n, d\ge 2$. Then the following statements are equivalent
\begin{enumerate}
\item $N_{d+1}(X, T)=N_{d+1}(X, T^n)$.

\item There is a dense $G_\delta$ subset $X_0$ of $X$ such that for any $l\in \Z$ and $x\in X_0$, there is a sequence $\{q_i\}$ of $\Z$
with
$$T^{nq_i}x\lra T^lx, T^{2nq_i}x\lra T^{2l}x, \ldots, T^{dnq_i}x\lra T^{dl}x.$$

\item $\Big\{x\in X: \overline{\O}(x^{(d)}, \tau_d(T))=\overline{\O}(x^{(d)}, \tau_d(T^n))\Big\}$
is a dense $G_\d$ subset of $X$.
\end{enumerate}
\end{thm}

\subsection{Proof of Theorem D}

\begin{proof}[Proof of Theorem D]
Let $(X,T^k)$ be minimal for some $k\ge 2$ and $d\in\N$. We show that for any $d\in \N$ and any $0\le j<k$ there is a sequence
$\{n_i\}$ with $n_i\equiv j\ (\text{mod}\ k)$ such that $T^{n_i}x\rightarrow x, T^{2n_i}x\rightarrow x, \ldots, T^{dn_i}x\rightarrow x,$ for a dense $G_\d$ subset of $X$. It is equivalent to show that
for any non-empty open subset $U$ of $X$ and $0\le j<k$
one has
\begin{equation}\label{Fur-linear-topo-finer}
U\cap T^{-n}U\cap \ldots \cap T^{-dn}U\not=\emptyset,
\end{equation}
for some $n\equiv j\ (\text{mod}\ k)$.

\medskip

By Theorem C and Lemma \ref{elibenjy}, there is a dense $G_\delta$ subset $X_0$ such that for any $x\in X_0$,
$$\{x\in X: \overline{\O}(x^{(d)}, \tau_d)=\overline{\O}(x^{(d)}, \tau_d^k)\}$$

Now assume that $0\le j<k$. Note that
$$(T^{-j}x,T^{-2j}x,\ldots, T^{-dj}x)\in \overline{\O}(x^{(d)}, \tau_d)$$ Thus,
for any non-empty open subset $U$ of $X$, there are $x\in U\cap X_0$ and $n_i\to \infty$ such that
$$T^{kn_i}x\lra T^{-j}x, T^{2kn_i}x\lra T^{-2j}x, \ldots, T^{dkn_i}x\lra T^{-dj}x,$$ i.e.
$$T^{kn_i+j}x\lra x,\ T^{2kn_i+2j}x\lra x, \ldots, T^{dkn_i+dj}x\lra x.$$ Set $n=kn_i+j$ when $i$ larger enough.
We then have
$$U\cap T^{-n}U\cap T^{-2n}U\cap \ldots \cap T^{-dn}U\not=\emptyset.$$
The proof is complete.
\end{proof}

\subsection{Proof of Theorem E}

\begin{proof}[Proof of Theorem E]
Let $U,V$ be non-empty open subsets of $X$. We are going to show that there is $n\in\N$ such that
$$U\cap T^{-P(n)}V\not=\emptyset.$$

Since $(X,T)$ is minimal, there is $N\in\N$ such that
$$X=\bigcup _{i=1}^N T^iU.$$
Let $q(n)=an^2+bn$.
By Bergelson-Leibman Theorem \cite{BL96} there are $n\in\N$ and $x\in V$ such that
$$T^{q(n)}x\in V,\  T^{q(2n)}x\in V, \ \ldots, \ T^{q(Nn)}x\in V.$$
Thus there is an open neighborhood $V_1$ of $x$ such that $V_1\subset V$ with
$$T^{q(n)}V_1\subset V,\ T^{q(2n)}V_1\subset V, \ldots, T^{q(Nn)}V_1\subset V.$$
By Theorem C there are $k_i,l_i\in\Z$  such that
$$(T^{2an}\times \ldots\times T^{2aNn})^{k_i}(T^{2an}\times \ldots\times T^{2an})^{l_i}(x,\ldots,x)\lra (Tx,\ldots, T^Nx)\in TV_1\times \ldots\times T^NV_1.$$
This implies that there are $k,l\in\Z$ such that
$$T^{2ank+2aln-1}x \in V_1,\ T^{4ank+2anl-2}x\in V_1,\ \ldots, T^{2aNnk+2anl-N}x\in V_1,$$
i.e.
$$T^{q(n)+2ank+2aln-1}x\in V,\ T^{q(2n)+4ank+2anl-2}x\in V, \ldots, T^{q(Nn)+2aNnk+2anl-N}x\in V.$$
For $j\in \{1,2,\ldots, N\}$,
\begin{equation*}
  \begin{split}
      & \quad q(jn)+2jank+2anl-j \\
       & = a(jn)^2+b(jn)+ 2jank+2anl-j\\
       & = a(jn+k)^2+b(jn+k)+c-(ak^2+bk+c)+2aln-j\\
       &= P(jn+k)-P(k)+2aln-j.
   \end{split}
\end{equation*}
Let
$$y=T^{-P(k)+2aln}x.$$
Then we have that
$$T^{P(jn+k)}(T^{-j}y)\in V, \ \forall j\in \{1,2,\ldots, N\}.$$
Since $X=\bigcup _{i=1}^N T^iU$, there is some $j_0\in \{1,2,\ldots, N\}$ such that
$T^{-j_0}y\in U$. Thus
$$T^{-j_0}y\in U\cap T^{-P(j_0n+k)}V.$$
In particular, $ U\cap T^{-P(j_0n+k)}V\neq \emptyset$. Then a standard argument by considering base of the topology of $X$
and taking the intersection yields the conclusion of the theorem.
The proof is complete.
\end{proof}

As a corollary we have

\begin{cor}\label{2-poly}Let $(X,T)$ be a totally minimal system, $k\ge 2$ and $0\le j<k$.
Let $P(n)=an^2+bn+c, a,b,c\in \Z, a\neq 0$ be an integral polynomials.
Then there is a dense $G_\delta$-subset $\Omega$ such that
for any $x\in \Omega$, $T^{P(n_i)}(x)\lra x$ for some sequence $\{n_i\}$ with $n_i\equiv j\ (\text{mod}\ k)$.
\end{cor}
\begin{proof}By putting $Q(n)=P(kn+j)$ and using Theorem E we get that there is a
sequence $\{n_i\}_1^\infty$ of $\N$  with $n_i\equiv j\ (\text{mod}\ k)$ such that
$$T^{P(n_i)}x\lra x$$ for $x$ in a dense $G_\delta$ set.
\end{proof}

\subsection{Proof of Theorem F}

\begin{proof}[Proof of Theorem F]
Let $(X,T)$ be a minimal system which is an open extension
of its maximal distal factor. We will prove that for any $d\in\N$, $\AP^{[d]}=\RP^{[d]}$. Since $\AP^{[d]}\subset \RP^{[d]}$, it suffices to show that $\RP^{[d]}\subset \AP^{[d]}$.

Let $d\in \N$ and $\pi_d:X\lra X_d$ be the factor to $X_d$. Since $(X,T)$ is an open extension
of its maximal distal factor, $\pi_d$ is open.  By Theorem \ref{main-distal},
there is a dense $G_\delta$ set $\Omega$ such that for each $x\in \Omega$,
$\overline{\O}(x^{(d+1)},\tau_{d+1})$ is $\pi_d^{(d+1)}$-saturated.

Let $x\in \Omega$ and $(x,y)\in \RP^{[d]}$. Then for each neighborhood $U$ of $y$,
there is $n\in\Z$ such that $T^{jn}x\in U$ for each $1\le j\le d+1$, which implies that
$(x,y)\in \AP^{[d]}$ by taking $x'=x$ and $y'=T^nx$ in the definition of $\AP^{[d]}$.

Now let $(x,y)\in \RP^{[d]}$. In each neighborhood $W$ of $(x,y)$, there are $(x',y')\in W$
with $x'\in \Omega$ and $(x',y')\in \RP^{[d]}$ by the openness of $\pi_d$. By what we just
proved, $(x',y')\in \AP^{[d]}$ which implies that $(x,y)\in \AP^{[d]}$ since $\AP^{[d]}$ is closed.
This ends the proof.
\end{proof}



\section{Some conjectures}

With the help of Theorem A and its consequences we
were able to
answer several open questions.
In fact, Theorem A also opens a window for the possibility to explore other natural questions, which
we will discuss now.

\begin{conj}\label{c-1-nil-1}
Let $(X,T)$ be a minimal nilsystem $d\in\N$. Then for each $k\ge 2$ the maximal $k$-step nilfactor factor of $(N_d(X), \G_d)$ is $(N_d(X_k), \G_d)$. Moreover, there is a dense $G_\delta$ set $\Omega$ such that for each $x\in \Omega$, the maximal $k$-step pro-nilfactor of $\overline{\O}(x^{(d)},\tau_d)$ is $\overline{\O}(\pi_\infty x^{(d)},\tau_d)$.
\end{conj}

We remark that in \cite{MR} Moreira and Richter showed among other things that if $X$ is a connected nilsystem, then for a.e. $x$
the Kronecker factor of $\overline{\O}(x^{(d)},\tau_d)$ is isomorphic to that of $X$. Assuming Conjecture \ref{c-1-nil-1} we have
\begin{conj}\label{f-coj-1}
Let $(X,T)$ be minimal and $d,k\in\N$. Then for each $d\ge 1$
the maximal $k$-step nilfactor factor of $(N_d(X), \G_d)$
is $(N_d(X_k), \G_d)$. Moreover, 
there is a dense $G_\delta$ set $\Omega$ such that for each $x\in \Omega$,
the maximal $k$-step pro-nilfactor of $\overline{\O}(x^{(d)},\tau_d)$ is
$\overline{\O}(\pi_\infty x^{(d)},\tau_d)$.
\end{conj}
\begin{proof}
This
follows by Theorems \ref{c-1-1},  \ref{c-1-1-distal}, and Conjecture \ref{c-1-nil-1}
\end{proof}


\begin{conj} Let $(X,T)$ be a totally minimal system, and $P(n)$ be a non-constant integral polynomial.
Then there is a dense $G_\delta$ subset $\Omega$ of $X$ such that
for every $x\in \Omega$, the set $\{T^{P(n)}(x):n\in\Z\}$ is dense in $X$.
\end{conj}
Remark that in Theorem E we have shown that it is true if $P(n)=an^2+bn+c$. Moreover, Conjecture 3 holds for minimal weakly mixing systems \cite{HSY-19}.

\medskip
We think that it
maybe relatively
easy
to show that Theorem D holds for any finite collection of non-constant integral polynomials $P_i(n)$
under the total minimality assumption (see Corollary \ref{2-poly}).
But, we believe it needs a real work to
verify
the following conjecture.

\begin{conj} Let $(X,T^k)$ be minimal for some $k\ge 2$ and $d\in\N$. Then for non-constant integral polynomials $P_m(n)$ with $P_m(0)=0$,
$1\le m\le d$, and any $0\le j<k$, there is a sequence $\{n_i\}_{i\in \N}$ such that
\begin{equation}\label{Fur-linear-topo-finer}
T^{P_1(n_i)}x\lra x, \ldots, T^{P_d(n_i)}x\lra x, \ i\to \infty,
\end{equation}
where $n_i\equiv j\ (\text{mod}\ k)$ and $x$ is in a dense $G_\delta$ set of $X$.
\end{conj}

The last conjecture is related to $\AP^{[d]}$.

\begin{conj} There is a minimal system $(X,T)$ with $\AP^{[2]}(X,T)=\Delta_X$ and at the same time, $(X,T)$ is not distal.
\end{conj}

\end{document}